\documentclass{amsart}
\usepackage[utf8]{inputenc}
\usepackage{amsmath}
\usepackage{amssymb}
\usepackage{amsthm}

\usepackage{booktabs}
\usepackage{mathrsfs}
\usepackage{enumerate}
\usepackage[labelfont=bf]{caption}
\usepackage{accents}
\usepackage{microtype}
\usepackage{mathtools}
\usepackage{longtable}
\usepackage{varwidth}

\usepackage{hyperref}
\usepackage{amsmath}
\usepackage{amsthm}
\usepackage{amssymb}
\usepackage{xfrac}
\usepackage{mathtools}
\usepackage{tikz-cd}
\usepackage{enumitem}
\usepackage{comment}
\usepackage{scalerel}
\usepackage{stackengine,wasysym}
\usepackage{csquotes}
\usepackage{todonotes}
\usepackage{cleveref}

\newcommand{\RR}{\mathbb{R}}
\newcommand{\ZZ}{\mathbb{Z}}

\newcommand{\NN}{\mathbb{N}}
\newcommand{\HH}{\mathbb{H}}
\newcommand{\EE}{\mathbb{E}}
\newcommand{\dd}{\mathrm{d}}

\newcommand{\MM}{\mathbb{M}}

\newcommand{\MP}{\mathbb{PM}}
\newcommand{\Points}{\mathbb{P}}
\newcommand{\Graphs}{\mathcal{G}}

\newcommand{\Complexes}{\mathcal{C}}
\newcommand{\GraphsU}{\mathcal{GU}}

\newcommand{\Hypers}{\mathcal{H}}
\newcommand{\HypersP}{\Points\Hypers}
\newcommand{\HypersV}{\mathcal{HV}}
\newcommand{\Riemanns}{\mathcal{R}}
\newcommand{\DelanGraph}{\mathcal{DG}}
\newcommand{\ThickGraph}{\mathcal{TG}}

\newcommand{\ThickHypersP}{\mathcal{THM}}
\newcommand{\DelanComplex}{\mathcal{DT}}
\newcommand{\ThickManifold}[2]{\mathcal{TM}(#1,#2)}

\newcommand{\vol}{\mathrm{vol}}
\newcommand{\Prob}[1]{\mathrm{Prob}\left(#1\right)}
\newcommand{\IRS}[1]{\mathrm{IRS}\left(#1\right)}

\newcommand{\DSubtf}[1]{\mathrm{DSub}_\mathrm{tf}\left(#1\right)}

\newcommand{\Ends}[1]{\mathcal{E}\left(#1\right)}
\newcommand{\EndsFin}[1]{\mathcal{E}_{<\infty}\left(#1\right)}
\newcommand{\EndsInf}[1]{\mathcal{E}_{\infty}\left(#1\right)}
\newcommand{\Isom}[1]
{\mathrm{Isom}\left(#1\right)}
\newcommand{\capacity}[2][]{\mathrm{cap}_{#1}\left(#2\right)}

\theoremstyle{plain}
\newtheorem{theorem}{Theorem}[section]
\newtheorem{lemma}[theorem]{Lemma}
\newtheorem*{lemma*}{Lemma}
\newtheorem{remark}[theorem]{Remark}
\newtheorem{prop}[theorem]{Proposition}
\newtheorem{cor}[theorem]{Corollary}
\newtheorem{example}[theorem]{Example}
\newtheorem*{theorem*}{Theorem}

\numberwithin{equation}{section}
\theoremstyle{definition}
\newtheorem{definition}[theorem]{Definition}

\title{Topology and dynamics of unimodular random hyperbolic manifolds}

\author{Ilya Gekhtman}
\address{Department of Mathematics, Technion -- Israel Institute of Technology, Haifa, Israel}
\email{gekhtman@technion.ac.il}

\author{Nir Lazarovich}
\address{Department of Mathematics, Technion -- Israel Institute of Technology, Haifa, Israel}
\email{lazarovich@technion.ac.il}

\author{Arie Levit}
\address{Department of Pure Mathematics, Tel Aviv University, Israel}
\email{arielevit@tauex.tau.ac.il}

\author{Asaf Nachmias}
\address{Department of Pure Mathematics, Tel Aviv University, Israel}
\email{asafnach@tauex.tau.ac.il}

\thanks{IG was partially supported by the ISF grant  3423/24.  NL was supported by the ISF grant  1576/23. AL was supported by the ISF grant  1788/22 and the BSF grant 2022105. AN was supported by the ERC consolidator grant 101001124 (UniversalMap) as well as ISF
grants 1294/19 and 898/23.}

\begin{document}

\begin{abstract}
We investigate the relationship between the space of ends of a unimodular random hyperbolic manifold and the dynamics of its geodesic flow. 
We show that having two ends of infinite volume implies recurrence, while having infinitely many such ends implies positive drift and entropy. We also provide a transience criterion which applies to deterministic hyperbolic manifolds. Our method relies on studying Delaunay graphs over point processes and on the analytic notion of capacity.
\end{abstract}

\maketitle

\section{Introduction}
To motivate the discussion, let us examine a certain relationship between dynamics and topology in group theory. Let $G$ be a finitely generated group and $\mu$ be a symmetric probability measure on $G$ whose support generates the group. Some of the dynamical properties of the associated random walk can be arranged in the following hierarchy.

\begin{enumerate}[label=(\alph*),leftmargin=*]
\item $G$ is \emph{finite}.
\item $G$ is \emph{recurrent}, 
 i.e almost every sample path of the random walk returns to the identity. This is equivalent to the group $G$ being finite or virtually $\ZZ$ or $\ZZ^2$.
\item $G$ has \emph{zero speed}, i.e. the distance in the word metric to the identity grows sublinearly along almost every sample path.
\item $G$ is \emph{Liouville}, i.e. every bounded $\mu$-harmonic function is constant. For example, if $G$ is virtually nilpotent then it is Liouville for any such measure $\mu$. 
\item $G$ is \emph{amenable}, i.e. the spectral radius is strictly less than one.
\end{enumerate}

The relevant topological invariant is the space of ends $\Ends{G}$ of the group $G$. By Stallings theorem $|\Ends{G}| \in \{0,1,2\}$ or  $\Ends{G}$ is the Cantor space. Clearly $\Ends{G} = \emptyset$ if and only if we are in case (a). If $|\Ends{G}|=2$ then $G$ is virtually $\ZZ$, which implies case (b). If $\Ends{G}$ is a Cantor space then the group $G$ is non-amenable, which implies the negation of case (e).  The remaining (and in some sense the most abundant) case $|\Ends{G}| = 1$ is harder to categorize.

We move on to another related classical setting of  regular covers.
Consider a compact $n$-dimensional hyperbolic manifold $M$ and a regular cover $N \to M$ with a group of deck transformations $G$. The relationship between the dynamical properties of the cover $N$ and the algebraic or analytic properties of the group $G$ was investigated extensively \cite{guivarc1980loi,brooks1981relation,lyons1984function,kaimanovich1986brownian,ledrappier2010linear}. We outline some of the classical theory in the following hierarchy.

\begin{enumerate}[label=(\Alph*),leftmargin=*]
    \item $N$ is \emph{compact}.  In that case the manifold $N$ is \emph{positively recurrent}, i.e. the geodesic flow is recurrent with finite expected return time (by Kac's formula).
    \item $N$ is \emph{recurrent}, in the sense that the geodesic flow is recurrent, i.e. for almost every $x \in T^1N$ the geodesic flow $\gamma_t$ satisfies $\liminf_t d_N(\gamma_tx,x) < \infty$ where $d_N$ is the metric on $N$. Recurrence is equivalent to the ergodicity of the geodesic flow by Hopf's dichotomy \cite{hopf1971ergodic,tsuji1959potential,sullivan1979density}. 
    In the infinite-volume case $N$ will be \emph{null recurrent}, i.e. the expected return time is infinite.
   \item $N$ has \emph{zero speed},  namely for almost every $x \in T^1N$ the geodesic flow $\gamma_t$ satisfies $\lim_t d_N(\gamma_tx,x)/t = 0$.
    
    \item $N$ is \emph{Liouville}, namely every bounded harmonic function on $N$ is constant. This is equivalent to the ergodicity of the $\pi_1(N)$-action on the Gromov boundary by the Furstenberg transform. 

    \item $N$ has \emph{Cheeger constant zero}, in the sense that $h(N)  = \inf \frac{\mathrm{Area}(\partial S)}{\mathrm{Vol}(S)} = 0$ where the infimum is taken over all compact smooth submanifolds $S \subset N$. This is equivalent to saying that $\lambda_0(N) = 0$ \cite{cheeger1970lower,buser1982note} or that the critical exponent is $\delta(\pi_1(N)) = n-1$ \cite{elstrodt1973resolvente, patterson1976limit}. 
\end{enumerate}
The dependence between all of the above properties in the settings of groups and regular covers with a given group of deck transformations is
$$
\begin{tikzcd}
(a) \arrow[r,Rightarrow] \arrow[d,Leftrightarrow] & (b) \arrow[r,Rightarrow] \arrow[d,Leftrightarrow] & (c) \arrow[r,Leftrightarrow] \arrow[d,Leftrightarrow] & (d) \arrow[r,Rightarrow] \arrow[d,Leftrightarrow] & (e) \arrow[d,Leftrightarrow] \\
(A) \arrow[r,Rightarrow] & (B) \arrow[r,Rightarrow] & (C) \arrow[r,Leftrightarrow] & (D) \arrow[r,Rightarrow] & (E)
\end{tikzcd}
$$

Our goal will be to extend this relationship going beyond regular covers. We immediately run into two problems. First, a general hyperbolic manifold is not associated to any particular covering group. We resolve this by replacing the covering group with a weaker topological invariant, which still retains some useful dynamical and geometric information, i.e. the space of ends of the manifold. For regular covers of compact manifolds,  this is the same invariant as the space of ends of the group of deck transformations. However, in the presence of cusps we need to be more careful and distinguish the subspace $\EndsInf{N}$ of infinite-volume ends  of the manifold $N$.

The second problem is that, for an arbitrary manifold $N$, we cannot expect to be able to say much based solely on its space of (infinite-volume) ends $\EndsInf{N}$. For example, consider the funnel $F = \left<\gamma\right> \backslash \HH^2$, where $\gamma$ is some loxodromic isometry of the hyperbolic plane $\HH^2$. Then $|\EndsInf{F}|=2$ but $h(F) > 0$. At the other extreme, by gluing pairs of pants with rapidly shrinking cuff lengths, one may construct a hyperbolic surface $S$ such that $\EndsInf{S}$ a Cantor space but $h(S) = 0$. 

We point out that the recurrence question (i.e. property (B) above) for hyperbolic surfaces is a very delicate and classical one, admitting numerous equivalent characterizations and reformulizations, see e.g. \cite{basmajian2022type} and the references therein. Some authors call it the \emph{type problem} \cite{ahlfors2015riemann}. Recurrent surfaces are said to be of \emph{parabolic type}. The class of all parabolic type surfaces is denoted $O_G$.
On its own, the space of ends is a hopelessly weak invariant for resolving the type problem. Interesting such examples, where the recurrence question boils down to the converge of certain series involving lengths of closed geodesics on the surface, include the two-ended flute surfaces \cite{basmajian1993hyperbolic} or infinitely-ended Cantor trees \cite{mcmullen1998hausdorff,vsaric2024quadratic,pandazis2024nonergodicity,bordenave2026quadratic}.

\subsection{Unimodular random hyperbolic manifolds}
To deal with this issue, we move from a deterministic to a random setting, and introduce a powerful additional assumption called unimodularity \cite{abert2022unimodular}. A random pointed hyperbolic manifold is a probability measure on the Gromov--Hausdorff space of hyperbolic manifolds. Such a random object is called \emph{unimodular} if, roughly speaking, its distribution \enquote{looks the same from all points} (see \S\ref{sec:unimodular}
for the formal definition). Unimodular hyperbolic random manifolds generalize regular covers. Our point of departure is an extension of Stallings' result to the unimodular setting by Biringer and Raimbault.

\begin{theorem*}[\cite{biringer2017ends}]
\label{thm:Biringer--Raimbault}
Let $\mu$ be a unimodular random hyperbolic manifold. Then $\mu$-almost every manifold $M$ has $|\EndsInf{M}| \in \{0,1,2\}$ or $\EndsInf{M}$ is a Cantor space.
\end{theorem*}

We remark that a unimodular hyperbolic manifold is essentially the same thing as a discrete invariant random subgroup of the group of isometries $\mathrm{SO}(n,1)$, see \cite[Theorem 1.4]{biringer2017unimodularity} and \cite[Theorem 2.8]{abert2022unimodular}, but we will not be pursuing the latter viewpoint here. Another remark is that the result of \cite{biringer2017ends} continues to hold for stationary random hyperbolic manifolds \cite{Kfir}. 

Our first main result is the  analogue of the above-mentioned fact that a two-ended finitely generated group is recurrent:
 
\begin{theorem}
\label{thm:main two ended}
Let $\mu$ be a unimodular random hyperbolic manifold. If $\mu$-almost every manifold has  two infinite-volume ends then $\mu$-almost every manifold is recurrent.
\end{theorem}

While the implications between the properties (A)--(E) are not true in complete generality for arbitrary manifolds, they do carry over to unimodular random manifolds (and more generally, to stationary random graphs \cite{benjamini2012ergodic} or manifolds \cite{lessa2016brownian}). For example, it is a consequence of our Theorem \ref{thm:main two ended} that a unimodular random hyperbolic manifold with two infinite-volume ends is \emph{amenable}, i.e. has zero Cheeger constant in the sense of property (E) above (this consequence can also be obtained directly). 

Our second main result is the analogue of the fact that an infinitely-ended finitely generated group is non-Liouville:

\begin{theorem}
\label{thm:main infinitely ended}
Let $\mu$ be a unimodular random hyperbolic manifold. If $\mu$-almost every manifold has infinitely many infinite-volume ends   then $\mu$-almost every manifold is non-Liouville. In particular, it has positive drift and positive volume growth.
\end{theorem}

The conclusion of Theorem \ref{thm:main infinitely ended} certainly implies that $\mu$-almost every manifold is transient (i.e. non recurrent), in distinction from the two-ended case considered in Theorem \ref{thm:main two ended}. We note that unlike the more strict case of regular covers (or finitely generated groups), a unimodular random hyperbolic manifold with infinitely many infinite-volume ends may almost surely be either amenable or non-amenable.

As part of the proof of both Theorems \ref{thm:main two ended} and \ref{thm:main infinitely ended} we are led to consider the interesting notion of an \emph{exit map}. This is a natural map $\zeta : T^1M \to \Ends{M}$ defined from the unit tangent bundle of a given \emph{transient} hyperbolic manifold to its space of ends. The properties of this map in the unimodular random setting are studied \S\ref{sec:ends and entropy}.

\begin{remark}
For unimodular random graphs with finite expected root degree, the analogues of Theorem \ref{thm:main two ended} and Theorem \ref{thm:main infinitely ended} were obtained 
in  \cite{angel2016unimodular} and 
\cite[Proposition 7.6]{angel2018hyperbolic} respectively.
\end{remark}

\subsection{Delaunay graphs and point processes}
For general unimodular random hyperbolic manifolds, including the ubiquitous one-ended ones, the situation is more complicated and not much can be said in general using our methods.
However, we do provide a criterion for transience (which can be regarded as a certain solution of the type problem) that applies in both the deterministic and the random setting: 

Given a hyperbolic manifold $M$ and a discrete point configuration $m$ on it,  we study the associated Delaunay graph $\DelanGraph(M,m)$. This is a graph whose vertices represent the Voronoi cells at the points of configuration $m$ and whose edges are geodesic arcs corresponding to neighboring Voronoi cells. We call the point configuration $m$ \emph{good} (Definition \ref{def:good points}) if its Delaunay graph enjoys certain geometric properties, most notably, that the angle between any pair of incident edges of the graph $\DelanGraph(M,m)$ is uniformly bounded away from zero.

We now state a deterministic result relating the transience of a hyperbolic manifold with the transience of a Delaunay graph on it.

\begin{theorem}
\label{thm:main:general transient condition}
Let $M$ be an $n$-dimensional hyperbolic manifold and $m$ a good point configuration on $M$. Equip the Delaunay graph $G = \DelanGraph (M,m)$  with transition probabilities proportional to the $(n-2)$-th powers of the distances in $M$. Namely,  for each vertex $v \in V(G) = m$ and edge $ \{v,u\} \in E(G)$  set the transition probability 
$$ p(v,u) = \frac{d_M(v,u)^{n-2}}{\sum_{\{v,w\} \in E(G)} d_M(v,w)^{n-2}}.$$
Then the manifold $M$ is a transient  if and only if the Delaunay graph $G$ is transient with respect to the random walk with transition probabilities given by $p$.
\end{theorem}


Theorem \ref{thm:main:general transient condition} is based on an analysis of capacities of graphs and of manifolds. It is inspired to a large extent by the ideas of \cite{gurel2017recurrence}, where a similar statement is obtained for planar surfaces (in dimension two). Theorem \ref{thm:main:general transient condition} allows us to reduce the transience problem for manifolds to an analogous one for graphs, which is typically easier to settle.

We address the problem of constructing good point configurations on hyperbolic manifolds and show that these always exist (Theorem \ref{thm:exists good point process}). The construction of a good point configuration is probabilistic (even for a given deterministic hyperbolic manifold). The idea is to make use of so-called \emph{controlled point processes} (Definition \ref{def:controlled point}). Roughly speaking, these are point processes whose intensity is comparable to the injectivity radius function on the manifold. A controlled point process is not hard to construct by a rather explicit probabilistic metric thinning method, and is sufficiently well-behaved geometrically to produce good point configurations (\S\ref{sec:controlled point process}).





\subsection{Asymptotic notations} We find it convenient to use the following asymptotic notations. We will write
$$ F \preceq G $$
to mean that $F \le C\cdot G$ for some implicit constant $C > 0$. We will write 
$$ F \approx G $$
to mean that both $F \preceq G $ and $G \preceq F $, namely, there is some implicit constant $C > 1$ such that $C^{-1} \cdot F \le G \le C \cdot F$.
Subscripts of the form $\preceq_n$ or $\approx_n$ will indicate that the constant implicit in the notation depends on $n$.

\newpage 

\subsection{Glossary}

Here is a glossary for some notations and symbols used in this work.


\begin{center}
\begin{longtable}{ll}
\hline
$(M,d_M,\vol_M)$ & a metric measure space $M$ with metric $d_M$ and measure $\vol_M$ \\
$\MM^n$ &  space of all $n$-pointed metric measure spaces \\
$\MM$ &  space of all pointed metric measure spaces (same as $\MM^1$)\\
\hline
$\Graphs$ &  subspace of pointed graphs in $\MM$\\
$\Complexes$ &  subspace of pointed simplicial complexes in $\MM$ \\
$\Hypers$ &  subspace of hyperbolic manifolds in $\MM$ \\
$\Riemanns$ &  subspace of Riemannian manifolds in $\MM$ \\
\hline
$L_\mu$/$R_\mu$ &  left/right measure defined on $\MM^2$ from $\mu\in\Prob{\MM}$  \\
$\mu_P$ &  measure obtained from $\mu\in\Prob{\MM}$ by conditioning on $P$\\
\hline
$\mathcal{U}(Q)$ & set of unbounded connected components of the complement of $Q$ \\
$\Ends{M}$ & space of ends of $M$ \\
$\EndsInf{M}$ & space of infinite-volume ends of $M$ \\
$\EndsFin{M}$ & space of finite-volume ends of $M$ \\
\hline
$\Points(M)$ & space of point configurations on the metric measure space $M$\\
$V_m(x)$ & Voronoi cell at the point $x \in m$ where $m \in \Points(M)$\\
$A_m(p)$ & index of Voronoi cell containing the point $p \in M$ where $m \in \Points(M)$\\
$\MP$ & space of pairs $(M,m)$ where $M \in \MM$ and $m \in \Points(M)$ is admissible \\
$\HypersP$ & same as $\MP$ with $M \in \Hypers$ being a hyperbolic manifold  \\
$\DelanGraph$ & Delaunay graph; formally, a map $\DelanGraph : \MP \to \Graphs$ \\
$\DelanComplex$ & Delaunay triangulation; formally, a map $\DelanComplex : \MP \to \Complexes$ \\
$l(e)$ & the length of the edge $e$ in the Delaunay graph\\
\hline
$\Prob{\Points(M)}$ & space of point processes on the metric measure space $M$\\
$\lambda_\theta$ &  intensity measure of the point process $\theta$\\
$\Pi_M$ & Poisson point process on the metric measure space $M$\\
\hline
$\mathrm{Prob}(\MP)$ & space of point processes over random metric measure spaces \\
$\Pi_\mu$ & Poisson point process over the unimodular metric measure space $\mu$\\
\hline
$\HH^n$ & $n$-dimensional hyperbolic space\\
$l_{x,y}$ & geodesic arc between the two points $x,y \in \HH^n$ \\
$\lambda_{x,y}$ & Lebesgue measure on the arc $l_{x,y}$ satisfying $\lambda_{x,y}(l_{x,y}) = d_{\HH^n}(x,y)$ \\
$\mathrm{InjRad}_M(x)$ & injectivity radius of the manifold $M$ at the point $x$\\
$\mathrm{cap}(K)$ & capacity of the subset $K$ in a graph or a manifold\\
$w$ & edge weight function controlling the $w$-transience of a graph\\
\hline
$B_v$ & ball of radius $r_v$ at the point $v$ on the hyperbolic  manifold $M$\\
$\nu_v$ & hyperbolic volume of the ball $B_v$\\
$\theta_{u,v}$ & law of random point on a random geodesic arc from $B_v$ to $B_u$\\
\hline
$f$ & function controlling a point process \\
$\lambda$ & fixed constant in the range $(\frac{1}{10},\frac{1}{6})$ used to scale the function $f$\\
$\mathcal{T}_M$ & thinning map for point configurations on the space $M$\\
\hline
$M_\textrm{thin}$/$M_\textrm{thick}$ & thin/thick part of the manifold $M$ \\
$\tau_M$ & infimum of injectivity radius over connected components of $M_\textrm{thin}$\\
$\ThickGraph$ & pointed thick graph construction map $\ThickGraph : \ThickHypersP\to\GraphsU$\\
$\ThickGraph(M,m)$ & thick graph on the manifold $M$ with point configuration $m$\\
$\ThickManifold{M}{m}$ &union of all Voronoi cells $V_m(x) \subset M$ for which $\tau_M(x) > 0$ \\
$\ThickHypersP$ & space of all $(M,m,p) \in \HypersP$ such that in addition $p \in \ThickManifold{M}{m}$\\
$\GraphsU$ & space of pointed $\left[0,1\right]$-edge-labeled graphs\\
\hline
$\HypersV$ & space of pointed hyperbolic manifolds with a unit tangent vector\\
$\gamma_t$ & geodesic action starting at some $(M,p,\vec v) \in \HypersV$ for all $t \in \mathbb{R}$\\
$\vec \mu$ & canonical lift a unimodular random hyperbolic manifold $\mu$ to $\HypersV$\\
$\HypersV^2$ & $\Hypers$ with two unit vectors tangent to the same geodesic\\
$L_{\vec \mu}, R_{\vec \mu}$ & left/right measures on $\HypersV^2$ arising from measure on the geodesic\\
$\DSubtf{\mathrm{SO}(n,1)}$ & Chabauty space of discrete torsion-free subgroups of $\mathrm{SO}(n,1)$\\
$\zeta_M$ & exit map on the hyperbolic manifold $M$ mapping $T^1 M$ to $\Ends{M}$\\
\hline
\end{longtable}
\end{center}

\section{Unimodular metric measure spaces}
\label{sec:unimodular}

In this preliminary section we introduce some key notions, including metric measure spaces, which is a natural framework to define and study unimodular random objects. We also discuss the space of ends associated to a metric space. 

For a much more detailed introduction to unimodular metric measures spaces, their properties and various examples, we refer the reader to \cite{khezeli2023unimodular}.

\subsection{Metric measure spaces}
\label{subsec:mms}

A metric measure space is essentially a space equipped with both a metric and a measure in a compatible manner. 
More precisely, we use the following  definition from \cite[\S3]{bowen2015cheeger}.

\begin{definition}
\label{def:metric measure space}
A \emph{metric measure space} is a triplet $(M,d_M,\vol_M)$ where $(M,\vol_M)$ is a  separable proper metric space and $\vol_M$ is a positive Radon measure on $M$. The measure $\vol_M$ can be finite or infinite.

A \emph{pointed metric measure space} is a quadruple $(M,d_M,\vol_M,p)$ such that $(M,d_M,\vol_M)$ is a metric measure space and $p \in M$ is a point. More generally, an \emph{$n$-pointed metric measure space} for some $n \in \NN$ is a tuple $(M,d_M,\vol_M,p_1,\ldots,p_n)$ such that $(M,d_M,\vol_M)$ is a metric measure space and $p_1,\ldots,p_n \in M$ are points.
\end{definition}

Let $\MM^0$ denote the space of all metric measure spaces.
Likewise,
let $\mathbb{M}^n$ denote the space of $n$-pointed metric measure spaces. It will be convenient for us to introduce the shorthand notation $\MM = \MM^1$ for the space of pointed metric measure spaces. In addition, whenever convenient  we will omit the explicit mention of the metric $d_M$ and the measure $\vol_M$, and write simply $(M,p)$. They are understood as being implicitly associated to the space $M$. 

We endow the space $\MM^n$ for each $n \in \NN$ with the topology described in \cite[Definition 5]{bowen2015cheeger}. This  is essentially a combination of the pointed Gromov--Hausdorff and the weak-$*$ topologies.

\begin{example}
\label{example:natural examples}
Here are some natural families of pointed metric measure spaces.

\begin{enumerate}
    \item Let $\Graphs \subset \MM$ consist of pointed locally finite graphs equipped with the graph metric and the counting measure on vertices. The  base point is a vertex of the graph.
        \item Let $\Complexes \subset \MM$ consist of pointed locally finite and finite-dimensional simplicial complexes  equipped with the intrinsic metric and the counting measure on vertices. The base point $p$ is a vertex of the complex.
\item Let   $\Hypers \subset \MM$ consist of pointed hyperbolic  manifolds  equipped with the hyperbolic distance and volume measure.
\item More generally, let $\Riemanns \subset \MM$ consist of pointed Riemannian manifolds equipped with the Riemannian distance and volume measure.
\end{enumerate}
These classes satisfy $\Graphs \subset \Complexes$ and $\Hypers \subset \Riemanns$.
\end{example}

\begin{remark}
All graphs we consider are locally finite.
\end{remark}

\begin{definition}
A \emph{random pointed metric measure space} is a Borel probability measure $\mu \in \Prob{\MM^1}$.
\end{definition}

For example, a random pointed graph/simplicial complex/hyperbolic manifold/Riemannian manifold is a Borel probability measure on $\Graphs$/$\Complexes$/$\Hypers$/$\Riemanns$ respectively.    
We now introduce one of the key notions for this work.

\begin{definition}
Given a random pointed metric measure space $\mu \in \Prob{\MM^1}$ we define two positive measures $L_\mu,R_\mu \in \mathrm{Meas}(\MM^2)$ by
\begin{equation}
\dd L_\mu(M,p,q) = \dd \mu(M,p) \dd \vol_M(q), \quad \dd R_\mu(M,p,q) = \dd \mu(M,q) \dd \vol_M(p)
\end{equation}
The random pointed metric measure space $\mu$ is \emph{unimodular} if $L_\mu = R_\mu$.
\end{definition}

Note that the positive measures $L_\mu$ and $R_\mu $ can in general be infinite.
A unimodular random pointed metric measure space will be referred to simply as a \emph{unimodular metric measure} space (or a \emph{unimodular graph}, \emph{manifold}, etc.). Unimodular graphs were first studied in \cite{aldous2007processes}. Vast literature has been dedicated to unimodular random graphs, see e.g. \cite{benjamini2012ergodic, BENJAMINI_LYONS_SCHRAMM_2015,salez2020spectral,lee2023relations,van2024number}. For rich information about unimodular random Riemannian manifolds we refer to \cite{abert2022unimodular}.

\begin{example}
Here are some examples of unimodular  metric measure spaces.
\begin{enumerate}
\item Let $(M,d_M,\vol_M) \in \MM^0$ be a metric measure space with $\vol_M(M) < \infty$. Selecting the point $p \in M$ uniformly at random with respect to $\vol_M$ and normalizing by $\frac{1}{\vol_M(M)}$ defines a unimodular metric measure space.

\item  Let $(M,d_M,\vol_M)$ be a metric measure space. Assume that $G \le \Isom{M}$ is a group of isometries admitting a Borel fundamental domain $F \subset M$ with $\vol_M(F) < \infty$. In this case, selecting a point $p \in F$ uniformly at random with respect to $\vol_M$ and normalizing by $\frac{1}{\vol_M(F)}$ defines a unimodular metric measure space.
\item Any Cayley graph (or more generally, any quasi-transitive graph) is a unimodular random graph.
\item Let $\HH^n$ be the $n$-dimensional hyperbolic space and $G = \Isom{\HH^n}$ so that $G \cong \mathrm{SO}(n,1)$. Fix a basepoint $x_0 \in \HH^n$ and denote $ K =\mathrm{Stab}_G(x_0)$. Let $\nu \in \IRS{G}$ be a discrete invariant random subgroup, namely $\nu$ is a $G$-invariant Borel probability measure on the space $\DSubtf{G}$ of discrete subgroups of $G$.  Let $\mu \in \Prob{\Hypers}$ be the pushforward of $\nu$ via the map $\Gamma \mapsto (\Gamma \backslash \HH^n, \Gamma x_0)$. Then $\nu$ is  a unimodular hyperbolic manifold. See \cite[\S3]{gelander2018invariant} for more details about this perspective.
\item Any weak-$*$ limit of unimodular metric measure spaces is  unimodular. 
\end{enumerate}
(Note that Example (2) has both Examples (1) and (3) as special cases).
\end{example}

When working with unimodular metric measure spaces, the following principle is extremely useful. Roughly speaking, it says that \enquote{the total mass a point sends out is the same as the total mass it receives}. 

\begin{lemma}[Mass transport principle]
\label{lemma:MTP}
A random point metric measure space $\mu$ is unimodular if and only if every Borel function $f : \MM^2 \to \left[0,\infty\right]$ satisfies
\begin{equation}
L_\mu(f) = R_\mu(f).
\end{equation}
\end{lemma}

Here is a simple but useful application of the mass transport principle, taken from \cite{aldous2007processes} (where it is established for graphs with the same proof).

\begin{lemma}[\enquote{Everything shows up at the root}]
\label{lemma: everything shows up at the root}
Let $\mu \in \Prob{\MM}$ be a unimodular metric measure space. Let $P$ be a measurable property of pointed metric measure spaces, i.e. a property defined on $\MM^1$. Then $\mu$-almost every pointed metric measure space $(M,p)$ has $P$ if and only if $\mu$-almost every pointed metric measure space $(M,p)$ is such that $(M,q)$ has $P$ for $\vol_M$-almost every $q \in M$.
\end{lemma}
\begin{proof}
Consider the measurable property $Q$ on the space $\MM^2$ of twice pointed metric measure spaces such that $(M;p,q)$ satisfies property $Q$  if and only if $(M,q)$ does not satisfy property $P$.

Note that  $R_\mu(Q)=0$ if and only if $\mu$-almost every $(M,q)\in \MM^1$ has property $P$.
Similarly, we have $L_\mu(Q) = 0$ if and only if $\mu$-almost every $(M,p)\in \MM^1$ is such that $\mathrm{vol}_M$-almost every point  $q \in M$ has property $P$. 
By the unimodularity of $\mu$, these two conditions are equivalent.
\end{proof}


The next lemma shows how to construct a new unimodular measure metric from a given one, by conditioning on a certain property and at the same time restricting the measure to the subspace where it holds.

\begin{lemma}[Conditioning]
\label{lemma:condional unimodular}
Let $\mu \in \mathrm{Prob}(\MM)$ be a unimodular metric measure space. Let $P$ be a measurable property of pointed metric measure spaces. For each 
$\mathcal{X} = (M,d_M,\vol_M,p) \in \MM$ define $T_P(\mathcal{X}) = (M,d_M,1_P\cdot\mathrm{vol}_M,p) \in \MM $. If $\mu(P) > 0$ then
$$\mu_P = \frac{(T_P)_* (1_P \cdot \mu)}{\mu(P)}  $$
is a unimodular metric measure space.
\end{lemma}
\begin{proof}
The fact that $\mu_P$ is unimodular follows by a direct application of the mass transport principle (Lemma \ref{lemma:MTP}). We leave the verification to the reader.
\end{proof}

A unimodular metric measure space  is \emph{extremal} if it cannot be written as a non-trivial convex combination of other unimodular metric measure spaces. Any unimodular  metric measure space $\mu$ can be represented as a convex combination $\mu = \int \theta \; \mathrm{d} \nu_\mu(\theta)$ where $\nu_\mu$ is a probability measure on $\mathrm{Prob}(\MM^1)$ such that $\nu_\mu$-almost every $\theta \in \mathrm{Prob}(\MM^1)$ is extremal. This is a consequence of Choquet's theorem \cite[Theorem \S3]{phelps2003lectures}.

\subsection{The space of ends}
Let us recall the classical topological notion of the space of ends. Roughly speaking, it is an invariant measuring all the different ways to go to infinity in a topological space. It was introduced by
Freudenthal in \cite{freudenthal1931enden}. It works well for the following class of spaces.

\begin{definition}
\label{def:Freudenthal}
A Hausdorff topological space $M$ is called \emph{Freudenthal} if $M$ is locally compact, locally connected, connected and separable. 
\end{definition}

Let $M$ be a Freudenthal topological space.   For each compact subset $Q \subset M$ let $\mathcal{U}(Q)$ denote the collection of unbounded (i.e. having non-compact closure) connected components of the complement $M \setminus Q$.  
The  sets $\mathcal U(Q)$ form an inverse system, in the sense that whenever a pair of compact sets satisfies $Q_1 \subset Q_2$ inclusion of connected components determines a map $\mathcal U(Q_2) \to \mathcal U(Q_1)$. 
Note that given a compact subset $Q \subset M$ the collection $ \mathcal U(Q)$ is finite \cite[Lemma 1.1]{raymond1960end}.

\begin{definition}
\label{def:space of ends}
The \emph{space of ends} $\Ends{M}$ of the Freudenthal space $M$ is the inverse limit of the inverse system $\mathcal U(Q)$ where $Q$ ranges over all compact subsets of $M$. An \emph{end neighborhood} of a given end $\zeta \in \Ends{M}$ is an element $V \in \mathcal{U}(Q)$ corresponding to $\zeta$ for some compact subset $Q \subset M$.
\end{definition}

Freudenthal proved that the space of ends $\Ends{M}$ of the Freudenthal space $M$ is a compact, totally disconnected, separable Hausdorff topological space with respect to the topology generated by end neighborhoods \cite{freudenthal1931enden}. In fact, it is the maximal compactification of the space $M$ with those properties\footnote{Further literature on the space of ends is \cite{
raymond1960end,
siebenmann1965obstruction,
peschke1990theory,
hughes1996ends,
guilbault2021endsshapesboundariesmanifold,
axon2025end,
bass2025ends}.}.

Assume further that $M$ is a Freudenthal metric measure space, equipped with the measure $\mathrm{vol}_M$. Let $\EndsInf{M}$ denote the subset of all \emph{infinite-volume ends}, namely ends $\zeta\in \Ends{M}$ all of whose end neighborhoods have infinite $\mathrm{vol}_M$-measure.
Let $\EndsFin{M}$ denote the subset of \emph{finite-volume ends}, namely $\EndsFin{M} = \Ends{M}\setminus \EndsInf{M}$. 
For example, a cusp of a hyperbolic manifold is a finite-volume end, and a funnel is an infinite-volume end. Ends of graphs always have infinite volume. 


To make a connection with unimodularity, we quote the following main result from \cite{biringer2017ends}.

\begin{theorem}[Biringer--Raimbault \cite{biringer2017ends}]
\label{thm:Biringer--Raimbault}
Let $\mu$ be a unimodular Freudenthal metric measure space. Then $\mu$-almost every space $M$ has
\begin{enumerate}
    \item either $|\EndsInf{M}| \in \{0,1,2\}$ or $\EndsInf{M}$ is a Cantor space, and
    \item if $\EndsFin{M} \neq \emptyset$ then $\overline{\EndsFin{M}} = \Ends{M}$.
\end{enumerate}
\end{theorem}

Biringer and Raimbault deal with unimodular random manifolds, however their proof is very robust; it applies verbatim to the more general setting of Freudenthal spaces (see also \cite{khezeli2023unimodular}). They are also able to obtain a much more detailed information in the special case of unimodular random surfaces. A similar classification of ends in the stationary case was recently obtained in \cite{Kfir}.

\section{Poisson point processes}
\label{sec:poisson point processes}

In this section we discuss point processes on metric measure spaces. An important and classical special case are Poisson point processes. We also consider Voronoi tessellations and Delaunay triangulations associated to such point processes. Lastly, we discuss point processes over unimodular random (rather than deterministic) metric measure spaces. 

\subsection{The Voronoi tesselation and Delaunay triangulation} \label{sub:voronoi}
Let $(M,d_M,\vol_M)$ be a metric measure space.
Recall that the Dirac measure $\delta_x$ for some $x\in M$ is the probability measure given by $\delta_x(A) = 1_A(x)$ for each Borel subset $A \subset M$.

\begin{definition}
The space $\Points(M)$ consists of all countable sums of Dirac measures on $M$. An element $m \in \Points(M)$ is called a \emph{point configuration}. A point configuration $m$ is called \emph{locally finite} if it has $m(A) < \infty$ for any bounded Borel subset $A \subset M$.
\end{definition}

Let $m \in \Points(M)$ be a point configuration on the metric measure space $M$. For a point $x \in M$  we will write $x \in m$ if $m(\{x\}) > 0$, i.e. the point $x$ is a part of the configuration $m$.

\begin{definition}
The \emph{Voronoi cell} at the point $x \in m$ with respect to the point configuration $m$ is the subset of $M$ given by
\begin{equation}
V_m(x) = \{z \in M \: : \: d_M(z,x) \le d_M(z,y) \quad \forall y \in m, y \neq x\}.
\end{equation}
\end{definition}

The point configuration $m \in \Points(M)$ is \emph{admissible} if it is locally finite and the Voronoi cells $V_m(x)$ at distinct points $x \in m$ are disjoint up to $\vol_M$-null sets.
For each $n \in \NN \cup \{0\}$ denote
$$ \MP^n = \{ (\mathcal X;m) \: : \: \mathcal X  \in \MM^n, \; m \in \Points(M) \text{ is admissible}\}.$$
We denote $\MP = \MP^1$. Namely $\MP$ is the space of metric measure spaces with a choice of a base point as well as an admissible point configuration.

Associated to an admissible point configuration $m \in \Points(M)$ there is a $\mathrm{vol}_M$-measurable map $A_m : M \to m$ sending $\mathrm{vol}_M$-almost every point $p \in M $ 
to the point $A_m(p) \in m$ determined by the condition that $p \in V_m(A_m(p))$. Roughly speaking, we can think of $A_m(p)$ as the index of the Voronoi cell containing the point $p$.

For the following definition, recall that  $\Complexes$ and $\Graphs$ denote respectively the space of locally finite simplicial complexes and graphs; see Example \ref{example:natural examples}.

\begin{definition}
\label{def:Delaunay triangulation}    
\emph{Delaunay triangulation} is a map 
\begin{equation}
\DelanComplex : \MP^0 \to \Complexes^0
\end{equation}
taking a metric measure space equipped with an admissible point configuration $(M;m) \in \MP^0$  to the simplicial complex $C = \DelanComplex(M;m) \in \Complexes^0$ defined as follows:
\begin{itemize}
    \item the set of vertices of $C$ is $\{x \in M \: : \: x \in m\}$, and
    \item the simplices of $C$ are finite collections of points $\{x_1,\ldots,x_n \: : \: x_i \in m\} $ for which there is some point $y \in M$ and some radius $r > 0$ such that $\mathring B_M(y,r) \cap m = \emptyset$ and $\{x_1,\ldots,x_n\} \subset \partial B_M(y,r)$.
\end{itemize} 

The Delaunay triangulation map extends to a pointed map $$\DelanComplex : \MP \to \Complexes$$ 
by taking $(M,p;m) \in \MP$ to the pointed grapg $(\DelanComplex^0(M;m), A_m(p)) \in \Complexes$ where  $A_m(p)$ is regarded as a vertex of the Delaunay triangulation. 

The \emph{Delaunay graph} map 
$$\DelanGraph : \MP \to \Graphs$$
is defined to be the $1$-skeleton of the Delaunay triangulation. Given an edge of a Delaunay graph $e \in \DelanGraph(M,p;m)$ incident at the two vertices $x,y \in m$ we denote its length 
by $l(e) = d_M(x,y)$.
\end{definition}

\begin{remark}
In general, the Delaunay triangulation is not a pure simplicial complex, namely its facets (i.e. maximal simplices) need not all be of the same dimension. 
\end{remark}

\begin{remark}
Let $m$ be an admissible point configuration $m$ on the metric measure space $M$. Note that the pointed map $M\to \Graphs$ given by $p\mapsto \DelanGraph(M,p;m)$ depends on the map $A_m$ and as such is only defined for $\vol_M$-almost every $p \in M$. 
\end{remark}

\begin{remark}
In certain nice geometric situations, such as Euclidean or hyperbolic spaces, the Delaunay graph admits an alternative geometric description. Namely, a pair of vertices $x,y \in m$ span an edge in the Delaunay graph if and only if the intersection of the two corresponding Voronoi cells $V_m(x)$ and $V_m(y)$ contains an open subset of some codimension-one hyperplane. More generally, the Delaunay triangulation can be characterized as being the \enquote{geometric dual} of the  Voronoi tesselation \cite[\S5]{deblois2018delaunay}. 
\end{remark}

\subsection{Point processes on metric measure spaces}
Let $(M,d_M,\vol_M)$ be a metric measure space.
\begin{definition}
A \emph{point process} on the metric measure space $M$
is a Borel probability measure $\theta \in \Prob{\Points(M)}$.
\end{definition}

In other words, a point process is simply a random point configuration. A point process $\theta$ is called \emph{locally finite} if $\theta$-almost every point configuration $m$ is locally finite. Locally finite point processes on $M$ can equivalently be viewed as random discrete subsets of $M$ counted with multiplicities.

A point process $\theta$ on $M$ determines an \emph{intensity measure} $\lambda_\theta$ on $M$ given by 
\begin{equation}
    \lambda_\theta(A) = \EE_\theta m(A)
\end{equation} 
where $A \subset M$ is any Borel subset, $m \in \Points(M)$ is distributed according to $\theta$ and $m(A)$ is to be understood as the number of points of the configuration $m$ belonging to $A$. We will restrict attention to point processes $\theta$ on the metric measure space $M$ of intensity given by $\lambda_\theta = \vol_M$.

\begin{example}
\label{example:poisson point process}
The \emph{Poisson point process} on the metric measure space $M$ with intensity $\vol_M$  is the unique point process $\Pi_M \in \Points(M)$ satisfying the following two conditions:
\begin{enumerate}
\item For every Borel subset $A \subset M$ the random variable $m(A)$ has a Poisson  distribution with parameter $\vol_M(A)$, and 
\item For every $n \in \NN$ and any collection $A_1,\ldots,A_n$ of disjoint Borel subsets of $M$ the random variables $m(A_1),\ldots,m(A_n)$ are independent 
\end{enumerate}
where the point configuration $m \in \Points(M)$ is distributed according to $\Pi_M$.
\end{example}

Indeed, for a locally finite point process $\Pi$  satisfying  $\lambda_\Pi=\vol_M$ and provided the measure $\vol_M$ has no atoms,  the two  properties (1) and (2)  above are in fact equivalent to each other \cite[Chapter 6]{last2017lectures}. In that case, a point configuration can be regarded of as a discrete subset of $M$ (no multiplicities needed).

A point process $\theta$ on $M$ is \emph{admissible} if $\theta$-almost every configuration $m$ is admissible in the sense introduced in \S\ref{sub:voronoi} above.

\subsection{Point processes on unimodular metric measure spaces}

We proceed to consider point processes defined over a random (rather than a deterministic) metric measure space.

\begin{definition}
\label{def:point process over unimodular}
Let $\mu$ be a unimodular metric measure space.
A \emph{point process} on $\mu$ 
is a probability measure $\theta \in \mathrm{Prob}(\MP)$ such that
\begin{itemize}
\item the pushforward of the measure $\theta$ to the space $\mathrm{Prob}(\MM)$ is $\mu$, 
\item the disintegration $\theta_{M,p} \in \mathrm{Prob}(\Points(M))$ of the measure $\theta$ over $\mu$-almost every fiber $(M,p) \in \MM$ is a locally finite point process on $M$ of intensity $\lambda_{\theta_{M,p}} = \mathrm{vol}_M$, and
\item the disintegration $\theta_{M,p} \in \mathrm{Prob}(\Points(M))$ of the measure $\theta$ over the  fibers $(M,p) \in \MM$ satisfies $\theta_{M,p} = \theta_{M,q}$
for $L_\mu$-almost every $(M,p,q) \in \MM^2$.
\end{itemize}
\end{definition}

We emphasize that, as part of the definition, the point process $\theta$ should be independent of the choice of basepoint.
The following is a technical but essential result.

\begin{prop}
\label{prop:there exists a Poisson point process}
Let $\mu$ be a unimodular  metric measure space. There exists a unique Poisson point process $\Pi_\mu$ over $\mu$, i.e. a point process disintegrating as a Poisson point process over $\mu$-almost every fiber $(M,p)$.
\end{prop}
\begin{proof}
On every metric measure space $M \in \MM^0$  there is a unique Poisson point process $\Pi_M \in \mathrm{Prob}(\Points(M))$ of intensity $\lambda_{\Pi_M} = \mathrm{vol}_M$ \cite[Theorem 3.6]{last2017lectures}. The point process $\Pi_M$ is locally finite. Associate to each pointed metric measure space $(M,p) \in \MM$  a measure $\Pi'_{(M,p)} \in\mathrm{Prob}(\MP)$ supported on elements of the form $(M,p;m)$  where the parameters $(M,p)$ are deterministic and the point configuration $m$ is distributed according to $\Pi_M$. Finally, take $\Pi_\mu = \int_{\MM} \Pi'_{(M,p)} \; \mathrm{d}\mu(M,p)$. The uniqueness of the process $\Pi_\mu$ follows from the uniqueness of each  Poisson point process $\Pi_M$.
\end{proof}

Assume that we are given an abstract map $\Phi$ which associates a graph to a point configuration on a unimodular metric measure space. Delaunay graphs are an example of such a map. We examine a certain convergence condition (i.e. $E < \infty$, see Equation (\ref{eq:convergence cond})) which allows to renormalize the resulting random graph making it unimodular. The notations $V_m$ and $A_m$ in the following statement are used in an abstract sense (broader than the Delaunay graph example).

\begin{prop}[Unimodular random graph from a point process]
\label{prop:pushing forward unimodular - general}
Let $\mu$ be a unimodular metric measure space and 
$\theta$ a point process over $\mu$. Let $\Phi^0 : \MP^0 \to \Graphs^0$ and $\Phi : \MP \to \Graphs$ be a pair of $\theta$-measurable maps such that  $\theta$-almost surely  
$$
\Phi(M,p;m) = (\Phi^0(M;m), A_m(p)) $$
where $A_m(p)$ is a vertex in the graph $\Phi^0(M;m)$.  For each $(M,p;m) \in \MP$ denote 
$$ V_m(p) = \{q \in M \: : \: \Phi(M,m;q) = \Phi(M,m;p) \}$$
so that $p \in V_m(p) \subset M$.
If
\begin{equation}
\label{eq:convergence cond}
E = \int_{\MP} \frac{1}{\vol_M(V_m(p))} \; \mathrm{d}\theta (M,m;p) < \infty
\end{equation}
then the probability measure on $\Graphs$ given by
\begin{equation}
\nu = \frac{1}{E} \cdot \Phi_* \, \left( \frac{1}{\vol_M(V_m(p))} \cdot \theta \right)
\end{equation} 
is a unimodular random graph. 
\end{prop}
\begin{proof}
We will show that $\nu$ is a unimodular random graph by verifying that it satisfies the mass transport principle (see Lemma \ref{lemma:MTP}).
Recall that $\Graphs^2$ denotes the space of twice pointed graphs. Let $h : \Graphs^2 \to \left[0,\infty\right]$ be any Borel function.  We have
\begin{align*}
L_\nu(h) &= \int_{\Graphs} \sum_{u \in V(G)} h(G,v,u) 
\; \mathrm{d}\nu(G,v) \\  
& = \frac{1}{E} \int_{\MP} \sum_{u \in \Phi^0(M,m)} \frac{h(\Phi^0(M,m),A_m(p),u)}{\vol_M(V_m(p))}
\; \mathrm{d}\theta(M,m;p) \\
&=\frac {1}{E} \int_\MM  \int_M f(M,p,q) \; \mathrm{d vol}_M(q)\; \mathrm{d}\mu(M,p) = \frac {1}{E} L_\mu(f)
\end{align*}
 where $f : \MM^2 \to \left[0,\infty\right]$ is the Borel function
 given by $$f(M,p,q) = \int_{\Points(M)}  \frac{h(\Phi^0(M,m),A_m(p),A_m(q))}{\mathrm{vol}_M(V_m(p)) \mathrm{vol}_M(V_m(q))}   \; \mathrm{d}\theta_M(m)$$
 and $\theta_M \in \mathrm{Prob}(\Points(M))$ is the disintegration of $\theta$ over $\mu$. The unimodularity of the measure $\mu$ implies that $L_\mu(f) = R_\mu(f)$, namely we may interchange the roles of the two points $p$ and $q$ in the last integral. By retracing the above equalities in the opposite order and with the roles of $p$ and $q$ reversed, we obtain $\frac {1} {E} R_{\mu}(f) = R_\nu(h)$. Altogether we get $L_\nu(h) = R_\nu(h)$. Hence the measure $\nu$ is unimodular.
\end{proof}

Delaunay graphs are a very natural situation where  Proposition \ref{prop:pushing forward unimodular - general}  directly applies. In that special case, $\Phi$ is the map $\DelanGraph$, the subset $V_m(p) \subset M$ is the Voronoi cell containing the point $p$, and $A_m(p) \in  V(\DelanGraph(M,m))$ is the vertex representing that cell. The following is a reformulation of Proposition \ref{prop:pushing forward unimodular - general} in that setting.

\begin{cor}
\label{cor:unimodular graph from unimodular space - delaunay}
Let $\mu$ be a unimodular metric measure space and $\theta$ a point process over $\mu$. If the measures of the Voronoi cells satisfy
\begin{equation}
E = \EE_\theta \frac{1}{\vol_M(V_m(p))} < \infty
\end{equation}
then the probability measure on $\Graphs$ given by
\begin{equation}
\nu = \frac{1}{E} \cdot \DelanGraph_* \, \left(\frac{1}{\vol_M(V_m(p))} \cdot\theta \right)
\end{equation} 
is a unimodular random graph. 
\end{cor}

Morally speaking, we would have liked to use Corollary \ref{cor:unimodular graph from unimodular space - delaunay} in the proof of Theorem \ref{thm:infinitely ended unimodular random surfaces are transient}.  This can be done directly for unimodular hyperbolic manifolds which admit a uniform lower bound on injectivity radius. In the general case however, we  have to resort to the more robust Proposition \ref{prop:pushing forward unimodular - general}, by considering a certain modified variant of Delaunay graphs called thick graphs (i.e. in order to make the expectation of $E$ finite, see \S\ref{sec:transience for manifolds via graphs} below for more details). So, formally speaking, we never have the opportunity to invoke Corollary \ref{cor:unimodular graph from unimodular space - delaunay} directly in this work. It is stated with the hope the reader may find it useful.


\begin{remark}
The assumption that the expectation of the quantity $\vol_M(V_m(p))^{-1}$ is finite is needed to normalize the resulting measure to be a probability. If this expectation is infinite then  one  naturally gets an \enquote{infinite unimodular measure}. We have decided not to consider such measures in this work.
\end{remark} 






\section{Transience for deterministic graphs and manifolds}
\label{sec:transience for deterministic}

In this section we consider deterministic (i.e. non-random) graphs and Riemannian manifolds. We use a combinatorial and an analytic notion of capacity to study transience of such spaces. In the graph case, we need to allow for transition probabilities which may vary according to edge weights.
Following \cite{gurel2017recurrence},  we introduce \emph{good graphs} on manifolds, which allow us to relate the transience problem for manifolds and for graphs --- see Theorem \ref{thm:transience passes to good graphs}.

\subsection{Capacity and transience for graphs}

Let $G$ be a connected locally-finite graph with vertex set $V(G)$ and edge set $E(G)$. Assume that $G$ is equipped with an \emph{edge weight} function $w : E(G) \to \mathbb{R}_{> 0}$. The function $w$ defines  transition probabilities $p_w$ on vertices given by
$$ p_w(u,v) = \frac{w(u,v)}{\sum_{\{u,v'\} \in E(G)} w(u,v')}$$
for each pair of vertices spanning an edge $ \{u,v\} \in E(G)$, and $p_w(u,v) = 0$ otherwise. The graph $G$ is called \emph{$w$-recurrent} or \emph{$w$-transient} respectively if it is recurrent or transient with respect to the random walk on its vertices induced by the transition probabilities $p_w$. In the special case where the edge weights $w$ are constant, we obtain the simple nearest neighborhood random walk, and recover the standard notions of a \emph{recurrent} or \emph{transient} graph.

We present a capacity criterion for $w$-transience. 
Let $C_0(X)$ denote the space of finitely-supported real-valued functions on a set $X$.
Let $ \nabla_w : C_0(V(G)) \to C_0(E(G))$ be the combinatorial gradient taking into account the edge weight function $w$. It is defined so that 
$$|\nabla_w f(e)| = w(e) \cdot |f(u)-f(v)|$$ 
where $e = \{u,v\} \in E(G)$ is an edge of the graph.
We consider the inner product
$$\left<f,g \right>_w = \sum_{e\in E(G)} \frac{f(e)g(e)}{w(e)}$$
on $C_0(E(G))$.
Given a function $f \in C_0(V(G))$ we have
\begin{equation}\label{eq: norm of nabla}
    \left<\nabla_w f, \nabla_w f\right>_w =  \sum_{e=\{u,v\}\in E(G)} w(e) \cdot (f(u) - f(v))^2.
\end{equation}

\begin{definition}
The \emph{capacity} $\capacity[w]{K}$ of a finite subset $K \subset V(G)$ is defined by
\begin{equation}
\capacity[w]{K} =  \inf_f \; \left< \nabla_w f, \nabla_w f\right>_w 
\end{equation}
where the infimum is taken over all finitely-supported functions $f : V(G) \to \left[0,1\right]$ satisfying $f(v) = 1$ for all $v \in K$.
\end{definition}

The following characterization is well-known (see e.g. \cite[Theorem 2.12]{woess2000random}).

\begin{theorem}
\label{theorem:capacity for graphs}
A connected locally finite graph $G$ with edge weights $w$ is $w$-recurrent is and only if $\capacity[w]{K} = 0$ for some (equivalently, all) finite subsets $K \subset V(G)$.
\end{theorem}

The following fact will be important.

\begin{remark}
\label{remark:transient subgraph}
Any connected locally finite graph with edge weights $w$ containing a $w$-transient subgraph is $w$-transient \cite[Corollary 2.15]{woess2000random}. 
\end{remark}

By a $w$-transient subgraph we mean a connected subgraph which is $w$-transient with respect to the restriction of the edge weight function $w$ to it.

\begin{lemma}
\label{lemma:bounded w}
Let $w_1$ and $w_2$ be two edge weight functions on the graph $G$. Assume that $w_1 \approx w_2$. Then the graph $G$ is $w_1$-transient if and only if it is $w_2$-transient.
\end{lemma}
\begin{proof}
Immediate (see \cite[Corollary 2.14]{woess2000random}).
\end{proof}

\subsection{Capacity and transience for manifolds}

Let $M$ be an $n$-dimensional Riemannian manifold with Riemannian volume $\mathrm{vol}_M$. The manifold  $M$ is called \emph{transient} if the Brownian motion starting at some point of $M$  eventually leaves some open subset of $M$ with positive probability (equivalently, the Brownian motion starting at any point of $M$ eventually leaves any precompact subset of $M$ with probability one). Otherwise, the manifold $M$ is called \emph{recurrent}. 
We mention a characterization of recurrence for manifolds in terms of the analytic notion of capacity. See \cite[\S4.3 and \S5]{grigor1999analytic} for this and many other equivalent characterizations.

\begin{definition}
\label{def:capacity of hyperbolic manifold}
The \emph{capacity} $\capacity{K}$ of a compact subset $K \subset M$ is defined by
\begin{equation}
\capacity{K} = \inf_f \, \int_M |\nabla f|^2 \; \dd \vol_M
\end{equation}
where the infimum is taken over all the compactly-supported locally Lipschitz\footnote{A function is \emph{locally Lipschitz} if each point admits a neighborhood on which the restriction of the function is Lipschitz (possibly with a non-uniform constant).} functions $f : M \to \left[0,1\right]$ with $f_{|K} = 1$.
\end{definition}

\begin{theorem}[{\cite[Theorem 5.1]{grigor1999analytic}}]
\label{theorem:capacity for surfaces}
The Riemannian manifold $M$ is recurrent with respect to  Brownian motion if and only if $\capacity{K} = 0$ for some (equivalently, all) compact subsets $K \subset M$.
\end{theorem}

\begin{remark}[Hopf dichotomoty \cite{hopf1971ergodic,tsuji1959potential,sullivan1979density}]
For hyperbolic manifolds, the transience of the Brownian motion is equivalent to the ergodicity of the geodesic flow.
\end{remark}

\subsection{Good point configurations}

Inspired by \cite{gurel2017recurrence}, we  introduce a notion of good point configurations on hyperbolic manifolds, and relate the transience of the associated Delaunay graph with that of the underlying manifold.

Given a hyperbolic manifold $M$ and a point $x \in M$ we let $\mathrm{InjRad}_M(x)$ denote the \emph{injectivity radius} of the manifold $M$ at the point $x$. It is defined to be the supremum of all radii $r > 0$ of balls on which the exponential map at $x$ is a diffeomorphism.

\begin{definition}
\label{def:good points} \label{good:points}
Let $M$ be an $n$-dimensional hyperbolic manifold and $m \in \Points(M)$ be an admissible point configuration. Let $G = \DelanGraph(M,m)$ and $T = \DelanComplex(M,m)$ respectively be the associated Delaunay graph and triangulation (the graph $G$ is the $1$-skeleton of the simplicial complex $T$). The point configuration $m$ is called \emph{$\varepsilon$-good} for some $\varepsilon > 0$ if it satisfies the following three properties:
\begin{enumerate}[label = (G\arabic*)]
\item \label{good: length} Any edge $e = \{v,u\} \in E(G)$ incident at some vertex $v \in V(G)$  satisfies $l(e) = d_M(u,v) < \min\{\frac{1}{4} \mathrm{InjRad}_M(v),1\}$.
 \item \label{good: triangulation} The Delaunay triangulation $T$ is indeed a triangulation of the manifold $M$. Namely, the map that sends each simplex of $T$ to the convex hull of its vertices is a homeomorphism between the topological realization of $T$ and the manifold $M$.
 
    \item \label{good: angle} 
    Any pair of distinct edges $e_1,e_2 \in E(G)$ incident  at some vertex $v \in  V(G)$ has angle $\angle_v(e_1,e_2) \ge \varepsilon$.
    
   
\end{enumerate}
We say that the point configuration $m$ is  \emph{good} if it is $\varepsilon$-good for some $\varepsilon>0$. 
\end{definition}



\begin{remark}    \label{remark:properties of good}
The following properties of the Delaunay graph $G = \DelanGraph(M,m)$ and triangulation $T = \DelanComplex(M,m)$ associated to a \emph{good} point configuration $m$ on the manifold $M$ follow directly from Definition \ref{def:good points}.
\begin{itemize}
    \item Every edge of the graph $G$ can be realized as an embedded geodesic arc in  the manifold $M$.
    \item \label{good:simplex}
    Every simplex of the complex $T$ is contained in a ball in $M$ isometric to a ball in the hyperbolic space $\mathbb{H}^n$, and can be realized as the convex hull of its vertices via this identification.
    \item There is some $k \in \NN$ such that the degree  of any vertex  $v \in V(G)$ satisfies $\deg (v) \le k$. In particular, the graph $G$ is locally finite.
    \item The hyperbolic law of sines implies that any two sides $e_1,e_2 \in E(G)$ of a given $2$-simplex in the complex $T$ satisfy $l(e_1) \approx _\varepsilon l(e_2)$. Combining this with the bound on vertex degrees, we deduce that in fact $l(e_1) \approx _\varepsilon l(e_2)$ for every pair of incident edges $e_1,e_2 \in E(G)$.
\end{itemize}
\end{remark}

Here is the main result of the current \S\ref{sec:transience for deterministic}, and one of the main results of the work:

\begin{theorem}
\label{thm:transience passes to good graphs}
Let $M$ be an $n$-dimensional hyperbolic manifold and $m$ be a good point configuration on $M$. Equip
the Delaunay graph $\DelanGraph (M,m)$ with the edge weight function given by $w_\mathrm{dist}(e) = d_M(u,v)$ for each edge $e = \{u,v\} \in E(G)$. 
The manifold $M$ is transient with respect to Brownian motion if and only if the Delaunay graph $\DelanGraph (M,m)$ is $w_{\mathrm{dist}}^{n-2}$-transient.
\end{theorem}

This is just a restatement of Theorem \ref{thm:main:general transient condition} of the introduction. Note that in the two-dimensional case the edge weight function is constant (since the exponent becomes $n-2=0$).
For the case of planar surfaces and so-called good graphs on them (not necessarily Delaunay), this is the main result of \cite{gurel2017recurrence}. Theorem \ref{thm:transience passes to good graphs} is a higher-dimensional generalization. 

\subsection{Some geometric lemmas}

The proof of Theorem \ref{thm:transience passes to good graphs} will be preceded by several lemmas, regarding simplices in Euclidean and hyperbolic geometry, and picking a uniformly random point on a random geodesic segment.

\begin{lemma}
\label{lemma:convex combination on simplex}
Let $\overline{\Sigma}$ be an $n$-dimensional simplex in Euclidean space with vertices $v_0,\ldots,v_n$.  Let $f : \{v_0,\ldots,v_n\} \to \RR$ be any function. Let $\widehat f : \overline{\Sigma} \to \RR$ be the convex interpolation of $f$ to the entire simplex $\overline{\Sigma}$, namely
$$ \widehat{f}(\sum_{i=0}^n t_i v_i) =\sum_{i=0}^n t_i f(v_i) \quad \forall t_i \ge 0, \sum_{i=0}^n t_i =1.$$
Assume that all angles between incident edges of $\overline{\Sigma}$ are lower bounded by some $\overline{\varepsilon} > 0$.
Then\footnote{As is commonplace in Riemannian geometry, we will use $|\nabla f|$ to denote the norm of the gradient $\nabla f$ with respect to the Riemannian metric.}
\begin{equation}
|\nabla \widehat f| \preceq_{\overline \varepsilon, n} \, \frac{\max_{i\neq j} |f(v_i) - f(v_j)|}{\min_{i\neq j}\|v_i-v_j\|}.
\end{equation}
\end{lemma}
\begin{proof}
Assume without loss of generality that $f(v_0)$ is a minimum of the function $f$ on the vertices of $\overline{\Sigma}$. For each $i \in \{1,\ldots,n\}$ denote
$$ e_i = \frac{v_i - v_0}{\|v_i-v_0\|}$$
so that $e_1,\ldots,e_n$ is a unit basis for Euclidean space. By the lower-bound on the angle between pairs incident edges of $\overline{\Sigma}$, the Jacobian of the change of basis from the basis $e_i$ to the standard basis is upper bounded by some constant $C'=C'(\overline{\varepsilon},n) > 0$. Note that 
$$ \frac{\partial \widehat{f}}{\partial e_i} = \frac{f(v_i) - f(v_0)}{|v_i-v_0|}.$$
The desired result now follows upon transforming the basis $e_1,\ldots,e_n$ to the standard basis and from the triangle inequality. Note that the gradient $\nabla \widehat f$ is constant on the simplex $\overline{\Sigma}$.
\end{proof}

We will use $\Sigma$ to denote  hyperbolic simplices and $\overline \Sigma$ to denote Euclidean ones.

\begin{lemma}
\label{lemma:transforming a hyperbolic simplex to a Euclidean simplex}
Let $\Sigma$ be an $n$-dimensional simplex in the hyperbolic space $\HH^n$.
Assume that all angles between incident edges of $\Sigma$ are lower bounded by some $\varepsilon > 0$. Then there is a  simplex $\overline{\Sigma}$ in Euclidean space and a $C$-bi-Lipschitz diffeomorphism  $f : \Sigma \to \overline{\Sigma} $. The angles between incident edges in the Euclidean simplex $\overline{\Sigma}$ are lower bounded by some constant $\overline{\varepsilon} > 0$. The constants $\overline{\varepsilon}$ and $C$   depend only on $\varepsilon$ and $n$.
\end{lemma}
\begin{proof}
By the hyperbolic law of cosines, the hyperbolic diameter of the simplex $\Sigma$ is upper-bounded by some constant $L >0$ which depends only on $\varepsilon$. Let $\mathcal{B}$ be the unit ball in the $n$-dimensional Euclidean space based at zero,  realizing the Klein ball model. We may isometrically (with respect to the hyperbolic geometry) embed the simplex $\Sigma$ in the model $\mathcal B$ so that it contains the point zero. In particular, $\Sigma$ will be contained in the hyperbolic ball of radius $L$ based at zero in the Klein model. This ball coincides with the Euclidean ball of radius $\tanh (L)$ based at zero. Note that the realization of $\Sigma$ in the Klein model is a Euclidean simplex $\overline \Sigma$. Let $f$ be the diffeomorphism realizing this identification (essentially, $f$ is the identity map on the realization of $\Sigma$ in the Klein model, mapping from hyperbolic to Euclidean geometry). By compactness, the operator norm of the differential $Df$ on all points of $\Sigma$ as well as of the inverse differential $Df^{-1}$ on all points of $\overline \Sigma$ is upper-bounded by some constant which depends only on $L$. The desired conclusions follow.
\end{proof}

Here is another geometric lemma about Delaunay triangulations arising from good point configurations. It is to be compared to the \enquote{sausage lemma} from \cite[Lemma 2.3]{gurel2017recurrence}.

\begin{lemma}[$n$-dimensional sausage lemma]
\label{lemma:sausage}
Let $M$ be an $n$-dimensional hyperbolic manifold and $m \in \Points(M)$ be an $\varepsilon$-good point configuration.
There is a constant $\delta= \delta(\varepsilon) < 1$ such that
any  pair of \emph{non-incident} edges $e_1,e_2 \in E(\DelanGraph(M,m))$ of 
the Delaunay graph satisfies $N_{\delta\cdot l(e_1)}(e_1) \cap N_{\delta\cdot l(e_2)}(e_2) = \emptyset$.
\end{lemma}

In the sausage lemma and its proof, we let $N_r(A)$ denote the $r$-neighborhood of the subset $A$, namely $N_r(A) = \{x \in M \: : \: d_M(x,A) < r\}$ for all  $r > 0$.

\begin{proof}[Proof of Lemma \ref{lemma:sausage}]
Let $G = \DelanGraph(M,m)$ and $T=\DelanComplex(M,m)$ denote  the Delaunay graph and triangulation respectively.
Consider some edge $e \in E(G)$ of the Delaunay graph. Let $A_e$ be the  union of all simplices in $T$ that share a vertex with $e$. 

By condition \ref{good: length}, we may regard  the edge $e$ and the subcomplex  $A_e$ as being isometrically embedded in the hyperbolic space  $\HH^n$ by lifting them to the universal cover. Regard the hyperbolic space  $\HH^n$ in the Klein model, parametrized so that the  midpoint of the edge $e$ coincides with the point $0$ in the model. The subset $A_e$ is contained in the hyperbolic neighborhood $N_1(e)$. Hence $A_e \subset B$ where $B$ is the ball of fixed hyperbolic radius $3$ (equivalently, fixed Euclidean radius $\tanh (3)$) around the point $0$. There is a constant $C > 1$ such that the hyperbolic and Euclidean geometries on the ball $B$ are $C$-bi-Lipschitz equivalent.

\Cref{remark:properties of good} implies that any edge $e' \in E(G)$ belonging to $A_e$ satisfies $l(e') \approx_\varepsilon l(e)$, where $l$ denotes the hyperbolic length of the edge. The above geometric comparison on $B$ shows that the Euclidean length of every edge $e' \in E(G)$ belonging to $A_e$ satisfies 
$
    \label{eq: bilip condition} l_{\mathbb{E}^n}(e)' \approx_\varepsilon l_{\mathbb{E}^n}(e).
$
By the same \Cref{remark:properties of good}, the number of edges and simplices in $A_e$ is bounded, hence the number of possible combinatorial types for the subcomplex $A_e$ is finite. 

Overall, the space of all possible Euclidean triangulations isometric to $A_e$ and considered up to homothety, is compact. Further, the subset $A_e$ contains an open neighborhood of the edge $e$ by condition \ref{good: triangulation}. 
By compactness, there is some $\delta_0 > 0$ such that the Euclidean neighborhood $N_{\delta_0 \cdot l_{\EE^n}(e)}(e)$ is contained in the interior of  $A_e$. 
By the above comparison of the hyperbolic and Eucludean geometries on $B$, there is a suitable constant $\delta = \delta(\varepsilon)> 0$ such that the hyperbolic neighborhood $N_{2\delta \cdot l(e)}(e)$ is contained in the interior of $A_e$ for all possible choices of an edge $e$ associated to some $\varepsilon$-good point configuration.

To conclude the proof, consider a pair of non-adjacent edges $e_1,e_2 \in E(G)$ in the Delaunay graph. Assume without loss of generality that $l(e_1)\ge l(e_2)$.
Note that $e_2$ is disjoint from the interior of $A_{e_1}$.
The above argument applied with respect to the edge $e_1$ shows that the neighborhood $ N_{2\delta \cdot l(e_1)}(e_1)$ is contained in the hyperbolic interior of $A_{e_1}$. Hence $e_2$ is disjoint from this neighborhood. The triangle inequality (applied with respect to hyperbolic geometry) implies $N_{\delta\cdot l(e_1)}(e_1) \cap N_{\delta\cdot l(e_2)}(e_2) = \emptyset$, as desired.
\end{proof}

In the next lemma, we show that picking a uniformly random point along a geodesic arc connecting two uniformly random points in small disjoint balls, determines a random point whose distribution is not overly distorted with respect to  hyperbolic volume (see Figure \ref{fig:balls}).
For a given pair of points $x,y \in \HH^n$ let $l_{x,y}$ denote the geodesic arc from $x$ to $y$ and $\lambda_{x,y}$ denote the Lebesgue measure supported on $l_{x,y}$ so that $\lambda_{x,y}(l_{x,y}) = d_{\HH^n}(x,y)$.

\begin{lemma}
\label{lemma:density of a random point on a uniform arc}
Fix a constant $0<c <\tfrac 1 3$. Let $u,v \in \HH^n$ be a pair of points with $r = d_{\HH^n}(u,v) \le 1$.  Fix a pair of radii $r_u,r_v$ satisfying 
$cr \le r_u,r_v \le \tfrac r 3$. Denote 
$$B_u=B_{\HH^n}(r_u), \; \nu_u = \mathrm{vol}_{\HH^n}(B_u) \quad \text{and} \quad 
B_v = B_{\HH^n}(r_v),\; \nu_v = \mathrm{vol}_{\HH^n}(B_v).$$
Then the probability measure
$$ \theta_{u,v} = \frac{1}{\nu_u \nu_v}  \int_{B_u\times B_v} \frac{\lambda_{x,y}}{\lambda_{x,y}(l_{x,y})} \; \mathrm{d} \mathrm{vol}_{\HH^n}(x,y)$$
is smaller (up to a universal multiplicative constant) than the uniform probability measure on the convex hull $W_{u,v} = \mathrm{conv}(B_u \cup B_v)$. Namely
$$\theta_{u,v} \le  \frac{C_n}{\mathrm{vol}_{\HH^n}(W_{u,v})} \cdot (\mathrm{vol}_{\HH^n})_{|W_{u,v}}$$
for some constant $C_n>0$ that depends only on $n$ and $c$. 
\end{lemma}

\begin{figure}
    \includegraphics[]{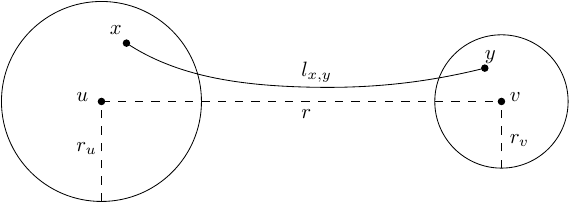}
    \caption{The measure $\theta$ in \Cref{lemma:density of a random point on a uniform arc} is the average of the $1$-dimensional normalized Lebesgue probability measures along the geodesic lines $l_{x,y}$ connecting points in $B_{\HH^n}(u,r_u)$ and $B_{\HH^n}(v,r_v)$.}
    \label{fig:balls}
\end{figure}

\begin{proof}
To begin with, we claim that the hyperbolic volume of the subset $W_{u,v}$ satisfies $$\vol_{\HH^n}(W_{u,v}) \approx_{c,n} r^n.$$ 
Indeed, note that
$B_{\HH^n}(u,cr)\subseteq W_{u,v}\subseteq B_{\HH^n}(u,3r)$. 
The hyperbolic and Euclidean metrics are $C$-bi-Lipschitz equivalent on this domain with a constant $C>1$ uniformly applicable to all $r \le 1$. The claim follows.

We will prove the lemma by providing an upper bound on the Radon-Nykodim derivative of the measure $\theta_{u,v}$ with respect to the hyperbolic volume measure. 
Let $z\in W_{u,v}$ be an arbitrary point and $\delta > 0$ be some sufficiently small radius. The hyperbolic volume of a ball of radius $\delta$ is proportional to $\delta^n$.  To bound the Radon-Nykodim derivative at the point $z$, it will suffice to show that
$$\theta_{u,v}(B_{\mathbb{H}^n}(z,\delta)) \preceq_{c,n}  \left(\tfrac{\delta}{r}\right)^n.$$

We proceed to establish this bound. Assume without loss of generality that $d_{\mathbb{H}^n}(z,v)\le d_{\mathbb{H}^n}(z,u)$.
Consider some point $x\in B_u$. 
All possible points $y\in B_v$ such that the geodesic arc $l_{x,y}$ intersects the ball $B_{\mathbb{H}^n}(z,\delta)$ belong to the hyperbolic cone $C_x(z,\delta)$ with apex $x$ passing through the ball $B_{\mathbb{H}^n}(z,\delta)$. 
Since $z$ is closer to $v$ than to $u$ and since $\delta$ is small, the angle $\alpha$ of the cone at its apex point $x$ satisfies $\alpha \preceq \frac{\delta}{r}$.
The volume of the intersection of the cone with the ball $B_v$ can be estimated by 
\begin{equation}\label{eq: vol of cone}
    \vol_{\HH^n}(C_x(z,\delta)\cap B_v)  \preceq  \alpha^{n-1}r^n \preceq  \delta^{n-1} r.
\end{equation}
To see this, note that $C_x(z,\delta)\cap B_v$ is contained in the intersection of the cone $C_x(z,\delta)$ with the ball $B_{\HH^n}(x,3r)$, and the volume of the latter intersection is at most $ \alpha^{n-1}r^n$ up to a multiplicative constant.

Finally, we require the estimates $\lambda_{x,y}(B_{\HH^n}(z,\delta))\le 2\delta$,  $\lambda_{x,y}(l_{x,y}) \ge r/3$, and  $\nu_u,\nu_v\succeq r^n$.
Putting everything together, we obtain for every point $x \in B_u$ that
$$ \frac {1}{\nu_v}\int_{ y\in B_v} \frac{\lambda_{x,y}( B_{\HH^n}(z,\delta))}{\lambda_{x,y}(l_{x,y})} \mathrm{dvol}_{\HH^n}(y) \preceq \frac{1}{r^n}\cdot \frac{6\delta}{r} \cdot  \delta^{n-1}r \preceq \left(\frac{\delta}{r}\right)^n $$
where the implicit multiplicative constants depend only on $n$ and $c$.
Taking the expectation of this expression with respect to a uniformly random point $x \in B_u$ gives 
$$ \theta_{u,v}(B_{\mathbb{H}^n}(z,\delta))=\frac {1}{\nu_u }\int_{x\in B_u} \left(\frac {1}{\nu_v}\int_{ y\in B_v} \frac{\lambda_{x,y}( B_{\mathbb{H}^n}(z,\delta))}{\lambda_{x,y}(l_{x,y})} \mathrm{dvol}_{\HH^n}(x,y) \right) \preceq_{c,n}  \left(\frac{\delta}{r}\right)^n $$
as desired.
\end{proof}

\subsection{Transience for hyperbolic manifolds and good Delaunay graphs}

We present the 
proof of the main result of the current section, by combining the two combinatorial and analytic capacity criteria for graphs and for manifolds, namely Theorems \ref{theorem:capacity for graphs} and \ref{theorem:capacity for surfaces}.

\begin{proof}[Proof of Theorem \ref{thm:transience passes to good graphs}]
Let $M$ be an $n$-dimensional hyperbolic manifold. Let $m \in \Points(M)$ be an $\varepsilon$-good point configuration on $M$. Let $G=\DelanGraph(M,m)$ and $T = \DelanComplex(M,m)$ respectively be the associated Delaunay graph and triangulation (i.e. the graph $G$ is the $1$-skeleton of the $n$-dimensional simplicial complex $T$). Regard  $G$ as being embedded inside $M$ with edges being geodesic arcs, and $T$ as a triangulation of $M$. 
Equip the graph $G$ with the edge weight function 
$ w : E(G) \to \mathbb{R}_{> 0}$ given by $$ w(e)=w_\textrm{dist}^{n-2}(e) = d_M(u,v)^{n-2} \quad \forall e=\{u,v\} \in E(G).$$

Let $\delta = \delta(\varepsilon) < 1$ be the constant provided by Lemma \ref{lemma:sausage}. For each vertex $v \in V(G) = m$ of the Delaunay graph denote 
$$r_v = \frac{\delta}{3} \cdot \min \{ l(e) : \text{$e\in E(G)$ is incident at $v$} \}$$
and $$B_v = B_{M}(v,r_v), \quad \nu_v=\vol_M(B_v)=\vol_{\HH^n}(B_{\HH^n}(r_v)).$$ 
Similarly, for each simplex $\Sigma$ of the Delaunay complex $T$ denote
$$ r_\Sigma = \min \{ l(e) : \text{the edge $e\in E(G)$ belongs to the simplex $\Sigma$} \}.$$
The geometric properties of $\varepsilon$-good point configurations imply that for each vertex $u \in V(G)$, edge $e = \{u,v\} \in E(G)$ incident at $u$ and simplex $\Sigma \in T$ containing $e$, the three quantities $r_u, r_\Sigma$ and $l(e)$ are all proportional up to  upper and lower universal multiplicative constants (depending on $n$ and $\varepsilon$), see Remark \ref{remark:properties of good}. In that situation, it follows that the quantities $r_v^{n-2}, r_\Sigma^{n-2}$ and $w(e)$ are likewise related. We will be using these observations repeatedly in the course of the proof.

\emph{Arguing in one direction.} Assume that the Delaunay graph $G$ is $w$-recurrent. Hence, according to Theorem \ref{theorem:capacity for graphs} the capacity of every finite subset of $V(G)$ is zero. Namely, there is a finitely-supported function $f : V(G) \to \left[0,1\right]$ with $f(v_0) = 1$ for some vertex $v_0 \in V(G)$ and with $\left< \nabla_w f, \nabla_w f\right>_w $ arbitrary small.

For each hyperbolic simplex $\Sigma \in T$ we apply Lemma \ref{lemma:transforming a hyperbolic simplex to a Euclidean simplex} and get a $C$-bi-Lipschitz diffeomorphism  $\varphi_\Sigma : \Sigma \to \overline{\Sigma}$ where $\overline{\Sigma}$ is a Euclidean simplex, in which all angles between incident edges are lower-bounded by $\overline{\varepsilon}$. Regard the function $f$ as being defined on the vertices of $\Sigma$ by identifying them with points in $V(G)$. Let $\overline f_\Sigma : \overline{\Sigma} \to \RR$ be the convex interpolation of the function $f \circ \varphi^{-1}_\Sigma$ from the vertices of $\overline{\Sigma}$ to its interior, as considered in Lemma \ref{lemma:convex combination on simplex}. Its gradient  is upper-bounded by 
$$ |\nabla  \overline f_\Sigma| \preceq \frac{1}{r_\Sigma}  \max_{\{u,v\}\in \Sigma_{(1)}} |f(u)-f(v)|$$
by Lemma \ref{lemma:convex combination on simplex}. 
Note that the Euclidean volume  of the simplex $ \overline{\Sigma}$ satisfies the upper bound  $\mathrm{vol}_{\EE^n}(\overline \Sigma) \preceq r_\Sigma^n$. Hence 
\begin{align*} 
\int_{\overline{\Sigma}} |\nabla  \overline f_\Sigma |^2 \; \mathrm{dvol}_{\EE^n} &\preceq  \frac{ \mathrm{vol}_{\EE^n}(\overline \Sigma) }{r_\Sigma^2} \cdot \max_{\{u,v\} \in \Sigma_{(1)}} | f(u)-f(v)|^2\\
&\preceq r_\Sigma^{n-2}\cdot \max_{\{u,v\} \in \Sigma_{(1)}} | f(u)-f(v)|^2.
\end{align*}
The pullback function $f_\Sigma = \overline f_\Sigma \circ \varphi_\Sigma: \Sigma \to \left[0,1\right] $ satisfies the similar  bound  
$$\int_\Sigma |\nabla f_\Sigma|^2 \; \mathrm{dvol}_{\HH^n} \preceq  \max_{e=\{u,v\} \in \Sigma_{(1)}} w(e) \cdot| f(u)-f(v)|^2$$
with a different implicit constant  (as the diffeomorphism $\varphi_\Sigma$ is $C$-bi-Lipschitz). 

Let $F : M \to \left[0,1 \right]$ be the function obtained by \enquote{piecing together} the functions $f_\Sigma$ on each simplex $\Sigma$, namely $F_{|\Sigma} = f_\Sigma$ for all simplices $\Sigma \in T$. If a pair of  simplices $\Sigma_1, \Sigma_2$ of $T$ shares a face then the  functions $f_{\Sigma_1}$ and $f_{\Sigma_2}$ coincide on the shared face. Note that the function $F$ is locally Lipschitz. Let $N \in \NN$ be the upper-bound on the number of simplices containing an edge of the Delaunay graph $G$ (such a bound exists because the Delaunay graph has bounded degrees; see Remark \ref{remark:properties of good}). 
Hence
\begin{align*}\|\nabla F\|_2^2&= \sum_{\Sigma\in T}\int_\Sigma|\nabla f_\Sigma|^2 \; \mathrm{d} \vol_M 
    \preceq  \sum_{\Sigma\in T}  \max_{e=\{u,v\} \in \Sigma_{(1)}} w(e)\cdot | f(u)-f(v)|^2\\
    &\le  N \cdot \left<\nabla_w f,\nabla_w f\right>.
    \end{align*}
where $\left<\nabla_w f,\nabla_w f\right>_w$ involves  the weight function $w$; see \Cref{eq: norm of nabla}.
This shows that  $\|\nabla F\|_2$ can be made arbitrary small. Moreover $F(v_0) = 1$ and $F$ is supported on a compact set. Therefore the capacity of the point $v_0 \in M$ is zero, and  the manifold $M$ is recurrent by Theorem \ref{theorem:capacity for surfaces}. 

\emph{Arguing in the converse direction.} 
Assume that the hyperbolic manifold $M$ is recurrent. Hence the capacity of any compact subset $K$ of $M$ is zero.
Fix an arbitrary basepoint $v_0 \in V(G)$ and denote $K = \overline{ B_{v_0}}$ (where $B_{v_0}$ is a ball at the vertex $v_0$, as defined in the start of the proof).
According to Theorem \ref{theorem:capacity for surfaces} there is a compactly-supported function $F : M \to \left[0,1\right]$ with $\|\nabla F\|_2$ arbitrary small and $F_{|K} = 1$. 
Define the function $f : V(G) \to \RR$  by
$$ f(v) = \frac{1}{\nu_v} \int_{B_v} F \; \mathrm{dvol}_{M}$$
for each vertex $v \in V(G)$.
In other words $f(v)$ is the expected value of the function $F$ on the ball $B_v$ around the vertex $v$.

We wish to upper bound the combinatorial gradient $\nabla_w f$  taking into account the edge weight function $w$ in terms of the hyperbolic gradient $\nabla F$. For each edge $e = \{u,v\} \in E(G)$ we have by Jensen's inequality (applied on the third line) and by the fundamental theorem of calculus (applied on the bottom line) that
\begin{align*} ( f(u) - f(v))^2 &= \left( \frac{1}{\nu_v} \int_{B_v} F(x) \; \mathrm{dvol}_{M}(x) - \frac{1}{\nu_u}  \int_{B_u} F(y) \; \mathrm{dvol}_{M}(y)\right)^2 \\
&= \left( \frac{1}{\nu_v \nu_u } \int_{B_v \times B_u} \left( F(x) - F(y) \right) \; \mathrm{dvol}_{M}(x,y)  \right)^2 \\
&\le \frac{1}{ \nu_v \nu_u }
 \int_{B_v \times B_u}
\left(  F(x) -  F(y) \right)^2 \;  \mathrm{dvol}_{M}(x,y)  \\ 
&\le \frac{1}{ \nu_v \nu_u }
\int_{B_v \times B_u}
\left(\int_{l_{x,y}}   |\nabla F| \; \dd \lambda_{x,y} \right)^2 \mathrm{d} \mathrm{vol}_M(x,y)
\end{align*}
where $l_{x,y}$ is the hyperbolic geodesic arc from $x $ to $y$ and  $\lambda_{x,y}$ is the Lebesgue measure on this geodesic arc normalized by arc length (i.e. $\lambda_{x,y}(l_{x,y}) = d_M(x,y)$). 
The triangle inequality and the geometric properties of good point configurations (discussed in the beginning of the proof) ensure that  
$$  d_M(u,v)-r_u-r_v \le \lambda_{x,y}(l_{x,y})  \le d_M(u,v) + r_u + r_v $$
and hence $\lambda_{x,y}(l_{x,y}) \approx d_M(u,v)$.
Plugging this into the above expression and using the Cauchy--Schwarz inequality
gives
\begin{align*}
( f(u) - f(v))^2  &\preceq  \frac{d_M(u,v)}{\nu_v \nu_u} \int_{B_v \times B_u}
 \int_{l_{x,y}} |\nabla F|^2 \; \dd \lambda_{x,y}  \; \mathrm{d} \mathrm{vol}_M(x,y) \\
 &\preceq \frac{d_M(u,v)^2}{\nu_v \nu_u } \int_{B_v \times B_u}
 \int_{l_{x,y}} |\nabla F|^2 \; \dd \frac{\lambda_{x,y}}{\lambda_{x,y}(l_{x,y})}  \; \mathrm{d} \mathrm{vol}_M(x,y).
 \end{align*}
 
 Consider the probability measure $\theta_{u,v}$ introduced in Lemma \ref{lemma:density of a random point on a uniform arc} with respect to the pair of points $v$ and $u$ and the  radii $r_u$ and $r_v$. 
The last integral in the above expression is in effect taken with respect to  $\theta_{u,v}$.  We  rewrite the above integral bound as 
$$ ( f(u) - f(v))^2 \preceq d_M(u,v)^2 \cdot \theta_{u,v}(|\nabla F|^2). $$
While Lemma \ref{lemma:density of a random point on a uniform arc} is stated in the contractible case (i.e. for $\HH^n$), it can be applied in the present situation, since the set $W_{u,v} = \mathrm{conv}(B_v \cup B_u)$ where the integral takes  place  embeds isometrically into $\HH^n$. It implies that 
$$ ( f(u) - f(v))^2 \preceq \frac{ d_M(u,v)^2}{\mathrm{vol}_M(W_{u,v})}  \int_{W_{u,v}} |\nabla F|^2\; \mathrm{d}\mathrm{vol}_M. $$
With respect to the combinatorial gradient $\nabla_w$ taking into account the edge weights $w$, we deduce that 
\begin{align*}
 \left<\nabla_w f,\nabla_w f\right>_w &= \sum_{ e=\{u,v\} \in E(G)} w(e) \cdot  (f(u)-f(v))^2  \\ 
 &\preceq \sum_{e= \{u,v\} \in E(G)} \frac{w(e) \cdot d_M(u,v)^2}{\mathrm{vol}_M(W_{u,v})}  \int_{W_{u,v}} |\nabla F|^2\; \mathrm{d}\mathrm{vol}_M  \\
 &= \sum_{ \{u,v\} \in E(G)} \frac{d_M(u,v)^{n} }{\mathrm{vol}_M(W_{u,v})}  \int_{W_{u,v}} |\nabla F|^2\; \mathrm{d}\mathrm{vol}_M.
 \end{align*}

At this point, we use the fact that the hyperbolic volume of the set $W_{u,v}$ is proportional to  $d_M(u,v)^n$ up to a multiplicative constant from above and below. Namely $d_M(u,v)^n \preceq \mathrm{vol}_M(W_{u,v})$. Further, by the convexity of the hyperbolic metric, the distance of any point in $W_{u,v}$ from the edge $e$ is at most  $\delta \cdot l(e)$.
The $n$-dimensional sausage lemma (i.e. Lemma \ref{lemma:sausage}) and the upper bound on the vertex degree in the Delaunay graph $G$ (Remark \ref{remark:properties of good}) imply that there is an upper bound $N \in \NN$ on the number of sets of the form $W_{u,v}$ which may overlap at any given point of the manifold $M$. 
Altogether, we get
 $$
  \left<\nabla_w f, \nabla_w f\right>_w \preceq  \sum_{ \{u,v\} \in G}\int_{W_{u,v}} |\nabla F|^2\; \mathrm{d}\mathrm{vol}_M
  \le  N \cdot \|\nabla F\|_2^2.
 $$
Hence $\left<\nabla_w f, \nabla_w f\right>_w$ can be made arbitrary small, by choosing a  function $F$ with $\|\nabla F\|_2$  small. Note that $f$ has finite support, $f(v) \in \left[0,1\right]$ for all $v \in V(G)$ and $f(v_0) = 1$. The shows that the capacity of the vertex $v_0$ in the Delaunay graph $G$ is zero. We conclude that the Delaunay graph $G$ is $w$-recurrent by Theorem \ref{theorem:capacity for graphs}.
\end{proof}

\section{Controlled point processes}
\label{sec:controlled point process}

The goal of this section is to construct by hand a good point process over any given deterministic  or  unimodular random hyperbolic manifold. The construction will take into account the injectivity radius function on the manifold.

\subsection{Controlled point configurations}

We develop a geometric tool that will assist in constructing good point configurations on hyperbolic manifolds with varying injectivity radius. The initial definitions apply to any metric space.

\begin{definition}
\label{def:controlled point}
Let $f : M \to \left(0,\infty\right) $ be a continuous function on a metric space $M$. 
A point configuration $m \in \Points(M)$ on the space $M$ is called  \emph{$f$-controlled} if
\begin{enumerate}[label =(\arabic*)]
    \item for every point $ x \in m$ of the configuration, the ball $B_M(x,f(x))$ contains no other point of $m$, and
    \item for every point $y \in M$ of the metric space, the ball $B_M(y,2 f(y))$ contains some  point of $m$.
\end{enumerate}
\end{definition}

An $f$-controlled point configuration is a generalization of an $\varepsilon$-net, in such a way that the density of the net is allowed to vary continuously over the space. Specifically, an $\varepsilon$-separated $2\varepsilon$-net is the same thing as an $f$-controlled point configuration with $f = \varepsilon$ constant. 



\begin{lemma}
\label{lemma:lower and upper bound on edges in controlled Delanauy graph}
Fix a constant $0 < \lambda < \tfrac{1}{2}$. Let $f : M \to \left(0,\infty\right)$ be a $\lambda$-Lipschitz  function on the  metric space $M$ and $m \in \Points(M)$ be an $f$-controlled admissible point configuration. Let $G = \DelanGraph(M,m)$ be the  associated Delaunay graph. Let  $e \in E(G)$ be an edge incident at some vertex $v \in V(G)$. Then the length $l(e)$ of the edge $e$ satisfies 
$$\frac{1}{2} f(v) \le l(e) \le \frac{4}{1-2\lambda}f(v).$$
\end{lemma}
\begin{proof}
The lower bound of $\frac{1}{2} f(v)$  on the edge length $l(e)$ follows directly from the definition of Voronoi cells and of the Delaunay graph, combined with the fact that $v $ is the only point of the configuration $m$ inside the ball $B_M(v,f(v))$. 

To get the upper bound, assume that $e = \{v,u\}$ for some other vertex $u \in V(G) = m$. By the definition of the Delaunay graph there exists a point $p \in M$ such that $r=d_M(p,u)=d_M(p,v)$ and $\mathring B_M(p,r)\cap m = \emptyset$.
By the triangle inequality  $r\ge d_M(u,v)/2=l(e)/2$. 
As the configuration $m$ is $f$-controlled we have $B_M(p,2f(p)) \cap m\neq \emptyset$. Therefore $r\le 2 f(p)$.
The fact that $f$ is $\lambda$-Lipschitz gives
$$ r\le 2 f(p)\le  2(f(v) + \lambda r).$$
Rearranging the above inequalities,  we deduce the desired upper bound
$$l(e)\le 2r\le \frac{4 f(v)}{1 -2\lambda}.\qedhere$$
\end{proof}

\subsection{Controlled point processes} 
A point process $\theta$ on a metric measure space $M$ is called \emph{$f$-controlled} if $\theta$-almost every point configuration $m \in \Points(M)$ is $f$-controlled.
In the current subsection, we   explicitly construct $f$-controlled point processes.
The idea of the construction is to apply an iterative \emph{metric thinning} procedure.

\begin{definition}[The thinning map]
\label{def:thinningMap}
Let $(M,d_M,\mathrm{vol}_M)$ be a metric measure space and $f : M \to (0,\infty)$ be a continuous function. \begin{enumerate}
    \item A point $x \in M$ is \emph{$(f,m)$-separated} with respect to the point configuration $m \in \Points(M)$ if 
$$d_M(x,y) \ge \max\{ f(x), f(y)\}$$ holds true for all points $y \in m $ distinct from $x$.
\item A point configuration $m \in \Points(M)$ is \emph{$f$-separated} if every point $x \in m$ is $(f,m)$-separated.
\end{enumerate} 
The \emph{thinning map} is the Borel map
$$ \mathcal{T}_M : \Points(M) \times \Points(M) \to \Points(M) $$
associating to a pair of point configurations $n,m \in \Points(M) $ a new point configuration $\mathcal{T}_M(n,m) \in \Points(M)$ consisting of all the $(f,n)$-separated points of $n$ together with all the $(f,n+m)$-separated points of $m$.
\end{definition}

The thinning map $\mathcal{T}_M$  satisfies the following properties:
\begin{itemize}
    \item The point configuration $\mathcal{T}_M(n,m)$ is $f$-separated for every $n,m \in \Points(M)$.
    \item If the point configuration $n \in \Points(M)$ is $f$-separated then $n \subset \mathcal{T}_M(n,m)$ for any $m \in \Points(M)$.
    \item The thinning map $\mathcal{T}_M$ is not symmetric in general.
\end{itemize}



\begin{lemma}
\label{lemma:1-Lipshitz is slowly varying}
Fix the constants $0 < \lambda < \frac{1}{2}$ and $\delta > 0$. Let $M$ be a metric measure space and $f : M \to (0,\infty)$ be a $\lambda$-Lipschitz function. 
Let   $n,m \in \Points(M) $ be a pair of point configurations and $x \in M$ be a point. If $n$ has no points in $B_M(x,2f(x) +\delta)$ and $m$ contains a \emph{unique} point $y \in B(x,2 f(x)+\delta)$ which moreover satisfies $y \in B_M(x,\lambda \delta)$ then $y\in \mathcal{T}_M(n,m)$.  
\end{lemma}

\begin{proof} 
Consider any point $z$ lying outside the ball $B_M(x,2f(x)+\delta)$. Hence the point $z$  satisfies $d_M(x,z) \ge 2f(x) + \delta \ge f(x) + \delta$. Furthermore, as the function $f$ is $\lambda$-Lipschitz, we have
$$d_M(x,z) \ge 2f(x) + \delta \ge 2(f(z) - \lambda d_M(x,z)) + \delta.$$
Rearranging the above inequality we obtain
$$ d_M(x,z) \ge \frac{2}{1+2\lambda} f(z) + \frac{1}{1+2\lambda} \delta \ge f(z) + \frac{1}{2} \delta.$$

Let $y$ be any point inside  the ball $B_M(x,\lambda \delta)$. Since the distance function $d_M(\cdot,z)$ is $1$-Lipschitz and the function $f$ is $\lambda$-Lipschitz, and since $\lambda<\frac{1}{2}$, we deduce from the above inequalities involving the point $x$ the following analogous inequality involving the point $y$:
$$ d_M(y,z) \ge \max \{f(y),f(z)\}.$$

Finally, consider the given pair of point configurations  $n,m \in \Points(M)$. If 
 $n$ has no points in $B(x,2f(x)+\delta)$ and $m$ contains a unique point $y\in B(x,2 f(x) + \delta )$ satisfying $y \in B(x,\lambda \delta )$ then  $d_M(y,z) \ge \max\{f(z),f(y)\}$ for any other point $z \in n + m$.  Hence $z\in \mathcal{T}_M(n,m)$ by the definition of the thinning map.
\end{proof}

\begin{prop}
\label{prop:there exists controlled point processes on mms}
Fix a constant $0 < \lambda < \frac{1}{2}$. Let $M$ be a metric measure space and $f : M \to (0,\infty)$ be a $\lambda$-Lipschitz function. Then $M$ admits an $f$-controlled point process $\nu$.
\end{prop}
\begin{proof}
Let $\mu_i$  be a sequence of independent Poisson point processes over the space $M$ of intensity $\mathrm{vol}_M$. The desired point process $\nu$ will be constructed iteratively.
First, let $\nu_0$ be the empty point process over $M$, i.e. $\nu_0 = \delta_\emptyset$ where $\emptyset$ is the empty point configuration.
Next, assume that the point process  $\nu_i$ has been constructed for some $i \in \NN \cup \{0\}$. 
The point process $\nu_{i+1}$ is defined by the following   formula
$$ \nu_{i+1} =  (\mathcal{T}_M)_* (\nu_{i} \times \mu_{i})$$
where $\mathcal{T}_M$ is the thinning map introduced above. In other words, the law of $\nu_i$ is obtained by applying the map $\mathcal{T}_M$ to a pair $(n,m)$ consisting of a $\nu_i$-random configuration $n$ and a $\mu_i$-random configuration $m$.
Finally, let $\nu$ be any weak-$*$ accumulation point of the sequence $\nu_i$ in the space $\Prob{\Points(M)}$.

We prove that the point process $\nu$ is $f$-controlled.
The properties of the thinning map $\mathcal{T}_M$ outlined below Definition \ref{def:thinningMap} ensure that $\nu$-almost every point configuration $m$ is $f$-separated, i.e. satisfies property (1) of Definition \ref{def:controlled point}. 

Next, let $x \in M$ be any point. We prove that $\nu$-almost every point configuration contains some point inside the ball $B_M(x,2f(x))$. Note that $f$-controlled point configurations are locally finite. Therefore for $\nu$-almost every point configuration $m$ the value $\delta_x(m)= \inf_{y\in m} d_M(x,y) -2f(x)$ is attained. Assume towards contradiction that $\delta_x(m) > 0$ with positive probability with respect to $\nu$. The same must be the case with respect to $\nu_i$ for all $i$ sufficient large. On the other hand, Lemma \ref{lemma:1-Lipshitz is slowly varying} dictates that upon using the  Poisson process $\mu_i$ to extend the point process  $\nu_i$ to  $\nu_{i+1}$  via the thinning map, there is some definite positive probability of inserting a new point inside $B_M(x,2f(x))$. This is a contradiction. The desired property (2) of Definition \ref{def:controlled point} follows by repeating the above argument with respect to some countable dense family of points $\{x_i\} \subset M$.
\end{proof}

\subsection{Controlled point processes are good}
We  show that an $f$-controlled point configuration on a hyperbolic manifold is good for a suitable function $f$.

\begin{definition}
Let $M$ be an $n$-dimensional hyperbolic manifold. An $f$-controlled point configuration $m \in \Points(M)$ is in \emph{general position} if for every point $x \in M$ and every radius $0 < r < 2f(x)$ we have $m(\partial B(x,r)) \le n+1$.
\end{definition}

The relevance of this notion of general position is that the Delaunay triangulation of such a point configuration is $n$-dimensional.

\begin{prop}
\label{prop:universal cover}
Fix a constant $0 < \lambda < \frac{1}{2}$. Let $M$ be an $n$-dimensional hyperbolic manifold. Consider the $\lambda$-Lipschitz function  $f : M \to \left(0,\lambda\right)$ given by 
$$
f = \lambda \cdot \min\{\tfrac{1}{4} \mathrm{InjRad}_M,1\}.
$$
Let $m \in \Points(M)$ be an $f$-controlled point configuration in general position and let $\widetilde m$ denote its lift to the universal cover $\HH^n$. All facets (i.e. maximal simplices) of the Delaunay triangulation $\DelanComplex(\HH^n,\widetilde m)$ are $n$-dimensional, and  $$\DelanComplex(M,m) = \pi_1(M) \backslash \DelanComplex(\HH^n,\widetilde m).$$
\end{prop}
\begin{proof}
Denote $\Gamma = \pi_1(M)$. The Delaunay triangulation $\DelanComplex(\HH^n,\widetilde m)$ is $\Gamma$-equivariant  as the group $\Gamma$ is acting on $\HH^n$ by isometries and preserves the point configuration $\widetilde m$. Let $p : \HH^n \to M$ be the universal covering map and denote $\widetilde f = f \circ p$. Since $f \le \frac{1}{8} \cdot\mathrm{InjRad}_M$ the map $p$ is a local isometry on any ball of the form $B_{\HH^n}(x,3\widetilde f(x))$. In particular the point configuration $\widetilde m$ is $\widetilde f$-controlled.
Recall that a finite collection of points $\mathcal{X} \subset \widetilde m$ spans a simplex in $\DelanComplex(\HH^n,\widetilde m)$ if there is a point $y \in \HH^n$ and a radius $r > 0 $ such that $\mathring B_{\HH^n}(y,r) \cap \widetilde m = \emptyset$ and $\partial  B_{\HH^n}(y,r) \cap \widetilde m = \mathcal{X}$. In this situation, the radius $r$ must satisfy $r < 2\widetilde f(y)$. It follows that $\mathcal{X}$ spans a simplex in $\DelanComplex(\HH^n,\widetilde m)$ if and only if $p(\mathcal{X})$ spans a simplex in $\DelanComplex(M,m)$. The fact that all facets are $n$-dimensional follows from the general position assumption. The desired conclusions follows.
\end{proof}

\begin{prop}
\label{prop:a Delanauy graph from a slowly varying controlled function is good}
Fix a constant $\frac{1}{10} < \lambda < \frac{1}{6}$. Let $M$ be an $n$-dimensional hyperbolic manifold. Consider the $\lambda$-Lipschitz function  $f : M \to \left(0,\lambda\right)$ given by 
$$
f = \lambda \cdot \min\{\tfrac{1}{4} \mathrm{InjRad}_M,1\}.
$$
Then every $f$-controlled point configuration $m \in \Points(M)$ in general position is $\varepsilon$-good for some constant $\varepsilon > 0$ depending only on $n$ and $\lambda$.
\end{prop}

\begin{proof}
Let $m \in \Points(M)$ be a point configuration and assume that $m$ is $f$-controlled and in general position. We need to verify that $m$ satisfies conditions \ref{good: length}, \ref{good: triangulation} and \ref{good: angle} from Definition \ref{good:points}.
First, according to Lemma \ref{lemma:lower and upper bound on edges in controlled Delanauy graph} an edge $e$ of the Delaunay graph $\DelanGraph(M,m)$ incident at the vertex $v$ is of length at most $$l(e) \le \frac{4}{1-2\lambda} f(v) \le \frac{4\lambda}{1-2\lambda}\cdot \frac{1}{4} \mathrm{InjRad}_M(v) \le \frac{1}{4} \mathrm{InjRad}_M(v).$$
This establishes condition \ref{good: length}. Note that we have used the bound $\lambda < \frac{1}{6}$.

Next, we establish condition \ref{good: triangulation}, which says that the Delaunay triangulation $\DelanComplex(M,m)$ topologically triangulates the manifold $M$. Consider the lift $\widetilde m$ of the point configuration $m$ to the universal cover $\mathbb H^n$. The $f$-controlled assumption ensures that  the discrete set $\widetilde m$ has a full limit set in $\mathbb H^n$. It is shown in \cite[Proposition 3.5]{deblois2018delaunay} that $\DelanComplex(\HH^n,\widetilde m)$ is a triangulation (in the standard topological sense) of the convex hull of the point configuration $\widetilde m $, which is the entire hyperbolic space $\HH^n$ in our case. Strictly speaking  \cite[\S3]{deblois2018delaunay} deals with finite point configurations, however the Delaunay triangulation of an $f$-controlled point configuration  depends locally only on finite subsets of $\widetilde m$. Condition \ref{good: triangulation} follows from Proposition \ref{prop:universal cover}.

Lastly, we establish condition \ref{good: angle}, namely a lower bound on the angle of incident edges in the Delaunay graph.
Assume towards contradiction that there is a sequence of point configurations $m_n \in \Points (M)$ and a sequence of triples of points $x_n,y_n,z_n \in m_n$ such that the pairs $\{x_n,y_n\}$ as well as $\{x_n,z_n\}$ span an edge in the Delaunay graph $\DelanGraph(M,m_n)$ with $\angle_{x_n}(y_n,z_n)\to 0$. 
By Proposition \ref{prop:universal cover} we may lift those configurations to the universal cover $\HH^n$, and assume implicitly that the discussion takes place there.
As the function $f$ is bounded, we may pass to a subsequence and assume that the limits $\phi = \lim_n f(x_n)$ and $\psi = \lim_n f(y_n)$ exist and are finite. It follows from Lemma \ref{lemma:lower and upper bound on edges in controlled Delanauy graph} that $\phi = 0$ if and only if $\psi = 0$. We will treat separately the cases where $\phi > 0 $ and $\phi = 0$.


Consider the first case where  $\phi,\psi > 0$. 
Up to passing to a further subsequence and renaming the points, we may assume without loss of generality that both limits $\delta_1 = \lim d_M(x_n,y_n) $ and $\delta_2 = \lim d_M(x_n,z_n) $ exist and that $\delta_1 \ge \delta_2$. Note that $\frac{\phi}{2} \le \delta_2 \le  \delta_1 \le \frac{4\phi}{1-2 \lambda}$ by Lemma \ref{lemma:lower and upper bound on edges in controlled Delanauy graph}. Let $q_n$ be the point on the geodesic arc $\left[x_n,y_n\right]$ at distance $\delta_2$ from $x_n$.  The angle assumption  $\angle_{x_n}(y_n,z_n)\to 0$ implies $d_M(z_n,q_n)\to 0$. As the configuration $m$ is $f$-controlled, we have $z_n \notin B(y_n,f(y_n))$. Passing to the limit, we deduce that $\delta_2 \le \delta_1 - \psi$. This leads to a contradiction to the fact that $x_n$ and $y_n$ span an edge in  the Delaunay graph. Indeed, any ball $B \subset M$ with $x_n,y_n \in \partial B$ will have $z_n \in \mathring B_n$ for all $n$ sufficiently large.

Consider the second case where $\phi,\psi = 0$. Taking into account Lemma \ref{lemma:lower and upper bound on edges in controlled Delanauy graph}, this means that $\lim_{n\to\infty} \mathrm{diam}_{\HH^n}(\{x_n,y_n,z_n\})= 0 $. Hyperbolic manifolds are approximately Euclidean on small scale. Hence we may rescale the point configurations in question via a Euclidean homothety to some definite scale, and repeat the above argument towards contradiction in the context of Euclidean geometry.

The fact that the constant $\varepsilon$ depends only on $\lambda$ and $n$ follows from a standard compactness argument.
 This concludes the proof.
\end{proof}


\subsection{Good point processes}

A point process $\theta$ on a hyperbolic manifold $M$ is called \emph{good} if $\theta$-almost every point configuration $m \in \Points(M)$ is $\varepsilon$-good for some fixed $\varepsilon > 0$.
We complete the picture by providing a probabilistic  construction of good point processes.

\begin{theorem}
\label{thm:exists good point process}
Every hyperbolic manifold admits a good point process.
\end{theorem}
\begin{proof}
Let $M$ be an $n$-dimensional hyperbolic manifold. Fix a constant $0 < \lambda < \frac{1}{6}$ 
and consider the
$\lambda$-Lipschitz function $f : M \to (0,\lambda)$ given by
$$
f = \lambda \cdot \min\{\tfrac{1}{4} \mathrm{InjRad}_M,1\}.
$$
The fact that $M$ admits an $f$-controlled point process was  established in a greater generality in Proposition \ref{prop:there exists controlled point processes on mms}. Further, in the context of that proof, the point process $\nu$ is obtained as a weak-$*$ limit of a sequence of point processes $\nu_i$. Each of those point processes $\nu_i$ can be regarded as being obtained from a Poisson point process of intensity $i\cdot \vol_M$ by deleting points, and as such is in general position. It follows that $\nu$ is in general position as well.
We conclude that the point process $\nu$ is good by Proposition \ref{prop:a Delanauy graph from a slowly varying controlled function is good}.
\end{proof}

Lastly, consider a unimodular random hyperbolic manifold $\mu$. We will say that a point process $\nu$ over the unimodular manifold $\mu$ is \emph{good}, if the fiber $\nu_{(M,p)} \in \mathrm{Prob}(\Points(M))$ over $\mu$-almost every pointed hyperbolic manifold $(M,p)\in\Hypers$ is a good point process on $M$.

\begin{cor}
\label{cor:exists good point process of unimodular}
Every unimodular random hyperbolic manifold admits a good point process defined over it.
\end{cor}
\begin{proof}
We know by Theorem \ref{thm:exists good point process} that every deterministic hyperbolic manifold admits a good point process. The desired point process over a given unimodular random hyperbolic manifold can be constructed by means of integration over fibers, in a similar manner to the proof of Proposition \ref{prop:there exists a Poisson point process}.
\end{proof} 

We point out that the good point processes constructed over a deterministic and a unimodular random hyperbolic manifold in Theorem \ref{thm:exists good point process} and Corollary \ref{cor:exists good point process of unimodular} respectively are in particular $f$-controlled with respect to the injectivity-radius related function $f$ considered above (see e.g. Proposition \ref{prop:there exists a good point process on a manifold}).

\section{Transience for unimodular random manifolds via graphs}
\label{sec:transience for manifolds via graphs}

In this section we investigate the transience of unimodular random hyperbolic manifolds. We do so by first constructing a good point process over the random manifold and then invoking the capacity criteria (i.e. Theorem \ref{thm:transience passes to good graphs}) to reduce the question to a transience problem for graphs. 
A delicate issue comes up in  implementing this strategy. Namely,  the resulting random good graph may not be unimodular. To overcome this problem we introduce an additional geometric object  we call the thick graph.

\subsection{The thick graph}


Let $\varepsilon_n > 0 $ be the Margulis constant for $n$-dimensional hyperbolic manifolds (see \cite[\S4.5] {thurston1997three} or \cite[\S8]{raghunathan}). For ease of notation, we assume without loss of generality that $\varepsilon_n < 1$. Every $n$-dimensional hyperbolic manifold $M$ can be decomposed into a thick part and a thin part, namely
$$M = M_\text{thin} \amalg M_\text{thick}$$
where
$$  M_\text{thin} = \{x \in M \: : \: \mathrm{InjRad}_M(x) \le \varepsilon_n\}\quad \text{and} \quad M_\text{thick} = \{x \in M \: : \: \mathrm{InjRad}_M(x) > \varepsilon_n\}.$$
By the thick--thin decomposition theorem in hyperbolic geometry \cite[Theorem 4.5.6]{thurston1997three}, the connected components of the thin part are of two distinct types:
\begin{itemize}
    \item neighborhoods of short geodesics, homeomorphic to a disc bundle over a circle, and
    \item neighborhoods of cusps, homeomorphic to a product of an $(n-1)$-dimensional Euclidean manifold with a half-infinite interval. Here, cusps are understood in a broader sense than the commonly used terminology, and can be either of finite or infinite volume. The only requirement is that the $(n-1)$-dimensional Euclidean factor is not contractible.
\end{itemize}

These two types of connected components of the thin part can be distinguished in that the injectivity radius is bounded away from zero in the first case, but not in the second case. This allows us to define a  function $\tau_M : M \to \left(0,1\right]$ by
$$ \tau_M(x) = \begin{cases}
\min\{\mathrm{InjRad}_M(x),1\} & x \in M_\text{thick} \\
\inf \left\{ \mathrm{InjRad}_M(y):
\begin{tabular}{c} $y$ is in the connected \\ component of $x$ in $M_\text{thin}$ \end{tabular}
\right\} 
&
x \in M_\text{thin}
\end{cases}
$$
for all $x \in M$.
By definition, the function $\tau_M$ is constant on each connected component of the thin part $M_\text{thin}$. This constant is positive on neighborhoods of short geodesics, and is zero on neighborhoods of cusps. 

Fix once and for all a constant $\frac{1}{10} < \lambda < \frac{1}{6}$. Consider the function
$$f_M : M \to (0,\infty), \quad f_M(x) = \lambda \cdot \min\{\tfrac{1}{4} \mathrm{InjRad}_M(x),1\}
  \quad \forall x \in M$$
  studied in the previous \S\ref{sec:controlled point process}.
 Let $m$ be an $f_M$-controlled point configuration on the $n$-dimensional hyperbolic manifold $M$, in the sense of Definition \ref{def:controlled point}. From this point onward, we will abbreviate and speak simply of \emph{controlled point configurations} (rather than $f_M$-controlled ones; it will always be understood with respect to the function $f_M$ on $M$). Recall that such a controlled point configuration always exists.

\begin{definition}
\label{def:thick graph}
The \emph{thick graph} $\ThickGraph(M,m)$ associated to the hyperbolic manifold $M$ and the point configuration $m$ is obtained from the Delaunay graph $ \DelanGraph(M,m)$ in two steps:
\begin{enumerate}
    \item Pass to the induced graph\footnote{An \emph{induced subgraph}  on a subset of vertices $V_0 \subset V(G)$ is the graph with vertex set $V_0$ and edge set $E_0 = \{ \{u,v\} \in E(G) \: :\:u,v \in V_0\}$.} $G_0$ on the set of vertices $v \in V(\DelanGraph(M,m))$ which do not belong to a connected component of the thin part which is a cusp neighborhood (equivalently, the set of vertices $v$ satisfying $\tau_M(v) > 0$).
    \item Take the quotient graph\footnote{A \emph{quotient graph} determined by an equivalence relation on $V(G)$ is a graph whose vertex set is $ \{\left[v\right] : v \in V(G)\}$ and whose edge set is $\{ \{\left[u\right],\left[v\right]\} : \left[u\right] \neq \left[v\right] ,\;\{u,v\}\in E(G)\}$.} of the  graph $G_0$, obtained by identifying vertices belonging to the same connected component of the thin part (at this point, necessarily a neighborhood of a short geodesic). Vertices belonging to the thick part are maintained (and are not identified with any other vertex).
\end{enumerate}
The resulting quotient graph is $\ThickGraph(M,m)$. 
The function $\tau_M$ naturally descends to a positive function on the set of vertices of the thick graph $\ThickGraph(M,m)$. 
\end{definition}

Note that the thick graph $\ThickGraph(M,m)$ associated to a controlled point configuration $m$ on the manifold $M$ is locally finite (that is why we had to remove vertices belonging to infinite-volume cusps neighborhoods, for such vertices would have introduced infinite degree to our graph).

Given a hyperbolic manifold $M$ and a controlled point configuration $m$, we denote
$$ \ThickManifold{M}{m}= \bigcup_{\substack{v \in m\\ \tau_M(v) > 0 }} V_m(v).$$
In other words,  $\ThickManifold{M}{m}$ is the union of all Voronoi cells corresponding to the point configuration $m$, based at those points of $m$ which do not belong to a neighborhood of a cusp (equivalently, points which do belong either to the thick part or to a neighborhood of a short geodesic). Very much roughly  speaking  $\ThickManifold{M}{m}$ looks like $M$ with all cusps removed (also known in the literature as a quotient of a \enquote{neutered} hyperbolic space). 

The following geometric result relates the space of infinite-volume ends of a hyperbolic manifold, with that of its thick graph.

\begin{lemma}
\label{lemma:properties of thick graph}
Let $M$ be a hyperbolic manifold and $m$ be a good and controlled point configuration on $M$. The thick graph $G = \ThickGraph(M,m)$  is locally finite, connected, and satisfies $|\Ends{G}| \ge |\EndsInf{M}|$.
\end{lemma}

\begin{proof}
The Delaunay graph $\DelanGraph(M,m)$ is locally finite by Remark \ref{remark:properties of good}. The thick graph $G$ is obtained from  a certain subgraph of $\DelanGraph(M,m)$ by identifying  vertices belonging to the same neighborhood of a short geodesic. The number of vertices of $\DelanGraph(M,m)$ in every such neighborhood is finite as the point configuration $m$ is controlled. Hence the thick graph $G$ is locally finite as well. 

We assume without loss of generality that the Margulis constant $\varepsilon_n$ is apriori taken to be sufficiently small, so that for any connected component $C$ of the thin part $M_\text{thin}$, the subset 
 $C_1 = \{x \in M \: : \: d_M(x,C) \le 1\}$ is also a neighborhood of a short geodesic or of a cusp, and that any two distinct such neighborhoods are disjoint.

Denote $\widehat M = \ThickManifold{M}{m}$ so that $\widehat M$ is a union of Voronoi cells based at points $v\in m$ such that $\tau_M(v)>0$. 
Let $C \subset M_\mathrm{thin}$ be any cusp of the hyperbolic manifold $M$  (of finite or infinite volume; see the above discussion on the thick-thin decomposition). 
Consider its $1$-neighborhood $C_1$ as above. 
We claim that the geodesic flow in the direction \enquote{away from the ideal point of the cusp} sets up a deformation retract of $C_1$ to $C_1\cap \widehat M$. 

Consider any point $z \in \partial C_1$. Since the point configuration $m$ is controlled we know that there is some point $v \in m$ with $d_M(z,v) \le \frac{1}{2}$ such that $z \in V_m(v)$. By the previous paragraph $v \in M_\text{thick}$ so that $z \in \widehat M$. Hence $\partial C_1 \subset \widehat M$.

Therefore, to establish the claim it will suffice to show that any geodesic ray $\gamma:[0,\infty)\to C_1$ with  $\gamma(0)\in \partial C_1 \subset \widehat M$ and escaping to the ideal point of $C_1$ intersects $\widehat M$ exactly at an interval of times of the form $\left[0,t_0\right]$ for some $t_0 \ge 0$. 
Assume to the contrary that the geodesic $\gamma$ crosses from some Voronoi cell $V_m(w)$ with $w\in M_\text{thin}$ to another Voronoi cell $V_m(v)$ with $v\in M_\text{thick}$.
The hyperplane $H$ separating the lifts of the two Voronoi cells $V_m(w)$ and $V_m(v)$ to the universal cover $\HH^n$ has the former and the ideal point of $C_1$ on one side, and the latter on the other side. However, the lift of the geodesic ray $\gamma$ travels from the lift of $V_m(w)$ to that of $V_m(v)$ and then back to the ideal point. This contradicts the convexity of the halfspaces of $H$. The above claim follows.

We conclude from the claim that  $\widehat M$ is a retract of $M$. 
In particular the subset $\widehat M$ is connected. 
In particular, any two vertices $v,u\in m$ with $\tau_M(v)>0$ and $\tau_M(u)>0$ can be connected by a path in the Delaunay graph which is contained entirely in $\widehat M$. This implies that the thick graph $G$ is also connected. 

We construct an injective map  $\EndsInf{M} \to \Ends{\widehat{M}}$. 
Consider some infinite volume end $\zeta \in \EndsInf{M}$. Let  $V \in \mathcal{U}_M(Q)$ be an end neighborhood of the end $\zeta $  corresponding to some sufficiently large compact subset $Q \subset M$ so that $Q \cap \widehat M \neq \emptyset$. The neighborhood $V$ cannot be contained in the union of the finite-volume cusps of $M$, since $V$ has infinite volume and any two cusps are a definite distance apart.
So  $V \cap C \neq \emptyset$ for some infinite-volume cusp $C$ of the manifold $M$. Therefore $V$ must have unbounded intersection with $\widehat M$, as $\partial C_1$ is unbounded and $C_1 \subset \widehat M$ by the above. 
There is some end neighborhood $W \in \mathcal{U}_{\widehat M}(Q \cap \widehat M)$ with $W \subset V$. These observations define a (non-canonical) injective map from $\EndsInf{M}$ to $\Ends{\widehat{M}}$.

Lastly, observe that $\Ends{G} \cong \Ends{\widehat M}$. This concludes the proof of the lemma.
\end{proof}

Let $\HypersP$ denote the space   of pointed hyperbolic spaces equipped with a point configuration (so that $\HypersP \subset \MP$). If a pointed manifold $(M,p;m) \in \HypersP$ happens to satisfy $p \in \ThickManifold{M}{m}$ then its base point $p$ descends to a base point of the corresponding thick graph $\ThickGraph(M,m)$.
This observation leads us to introduce the notation
$$ \ThickHypersP = \left\{(M,p;m) \in \HypersP \: : \: p \in \ThickManifold{M}{m} \right\}.$$
In addition, let $\GraphsU$ denote the space of pointed $\left[0,1\right]$-edge-labeled graphs, namely
$$\GraphsU = \{ (G,p;f) \: : \: (G,p) \in \Graphs, \; f : E(G) \to \left[0,1\right] \}.$$
The thick graph construction extends to a pointed map
$$ \ThickGraph : \ThickHypersP \to \GraphsU.$$
In other words, provided $p \in \ThickManifold{M}{m}$ we regard the thick graph $\ThickGraph(M,m)$ as a pointed $\left[0,1\right]$-edge-labeled graph, where each edge $\{\left[u\right],\left[v\right]\}$ has the label $f(e)=\min \{\tau_M(u),\tau_M(v)\} > 0$.  Roughly speaking, the label of an edge $e$ encodes the injectivity radius locally at $e$ (or the infimum of the injectivity radius over the connected component of the thin part represented by a vertex incident at $e$).





\begin{prop}
\label{prop:modified point process}
Let $\mu$ be a unimodular random hyperbolic manifold and 
$\nu$  a controlled point process over $\mu$.
\begin{enumerate}
    \item $\nu(\ThickHypersP)>0$.
    \item Let $\widehat{\nu}$ be the probability measure obtained from $\nu$ by conditioning on  the subset $\ThickHypersP$ in the sense of  \Cref{lemma:condional unimodular}. Then $\widehat{\nu}$ is a unimodular random metric measure space.
    \item There exists a Borel function $\psi : \GraphsU \to \RR_{>0}$ such that the $\left[0,1\right]$-edge-labeled random graph
    $\lambda = \psi \cdot \ThickGraph_* (\widehat{\nu})$ is unimodular.
\end{enumerate}
\end{prop}
The controlled point process $\nu$ over the unimodular random manifold $\mu$ exists  by  Corollary \ref{cor:exists good point process of unimodular}. 
\begin{proof}[Proof of Proposition \ref{prop:modified point process}]
The thick part $M_\text{thick}$ of any hyperbolic manifold $M$ is non-empty. Furthermore, we assume without loss of generality that the Margulis constant $\varepsilon_n$ is apriori taken to be sufficiently small, so that any two connected components of the thin part $M_\text{thin}$ are at distance $2$ apart. 
Hence, every controlled point configuration $m$ on the manifold $M$ must satisfy $m(M_\text{thick})>0$ (i.e. some points of $m$ lie in the thick part of $M$). Therefore $\vol_M(\ThickManifold{M}{m}) > 0$. By the \enquote{everything shows up at the root} principle (Lemma \ref{lemma: everything shows up at the root}) we get $\nu(\ThickHypersP)>0$. This shows part (1). Part (2) is just a special case of \Cref{lemma:condional unimodular}. Finally, part (3) follows directly from Proposition \ref{prop:pushing forward unimodular - general}, applied with respect to the thick graph construction map $ \ThickGraph : \ThickHypersP \to \GraphsU$. For a vertex $\left[v\right] \in \ThickGraph(M,m)$ the function $\psi$ is given by $$\psi(\left[v\right]) = \frac{1}{\vol_M\left(\bigcup_{x \in \left[v\right]} V_m(x)\right)}.$$
The function $\psi$ is upper-bounded. Indeed, consider a vertex $\left[v\right] \in \ThickGraph(M,m)$. If $v \in M_\mathrm{thick}$ then $\left[v\right] = \{v\}$ and the volume of the Voronoi cell at $v$ is lower-bounded away from zero. Otherwise  $\left[v\right]$ represents a neighborhood of a short geodesic. In that case there must be a point $x \in \left[v\right] \subset m$ with $\mathrm{InjRad}_M(x) \in (\varepsilon_n-\delta, \varepsilon_n)$ for some fixed constant $\delta > 0$. The volume of the Voronoi cell at the point $x$ is lower-bounded as well. It follows that $\int \psi \,\mathrm{d}\nu < \infty$ and the assumption of Proposition \ref{prop:pushing forward unimodular - general} is satisfied.
\end{proof}





\subsection{The infinitely-ended case}

We are ready to prove one of our key  results. It serves as a first step towards  Theorem \ref{thm:main infinitely ended} of the introduction. Our strategy will be to consider the relationship between  Delaunay and  thick graphs constructed from a controlled point process.

\begin{theorem}
\label{thm:infinitely ended unimodular random surfaces are transient}
Let $\mu$ be a unimodular random hyperbolic manifold. If $\mu$-almost every manifold has infinitely many infinite-volume ends then $\mu$-almost every manifold is transient with respect to Brownian motion.
\end{theorem}

The proof Theorem \ref{thm:infinitely ended unimodular random surfaces are transient} will make use of the following probabilistic lemma, which depends on well-known results in the theory  unimodular random graphs and trees; see e.g. \cite{aldous2007processes,lyons2017probability}. Some of those results are stated in the context of a percolation over a fixed Cayley graph; however the same proof applies more generally to a unimodular random graph.

\begin{lemma}
\label{lemma:finding trees}
Let $\lambda \in \mathrm{Prob}(\GraphsU) $ be a unimodular random $\left[0,1\right]$-edge labeled graph. Assume that $\lambda$-almost every graph is infinitely ended and has positive labels. Then for any sufficiently small threshold $0 < s < 1$ the connected component of the base point $p$ in the subgraph of $G$ obtained by removing all edges $e \in E(G)$ with edge label $f(e) < s$ is transient for a positive $\lambda$-measure of random edge-labeled graphs $(G,p;f) \in \GraphsU$.
\end{lemma}

Here, by \enquote{transient} we mean with respect to the simple nearest neighborhood random walk (irrespective of the edge labels).

\begin{proof}[Proof of Lemma \ref{lemma:finding trees}]
According to \cite[Lemma 7.7]{lyons2017probability} a $\lambda$-random pointed graph admits a random subtree containing the base point, and this subtree is infinitely ended with positive probability. Let $\eta$ denote the probability distribution of this random subtree. Write $\eta = t \eta' + (1-t) \eta''$ where $ t \in \left[0,1\right]$ and $\eta'$ is an \emph{infinitely-ended} unimodular random tree. The expected degree with respect to $\eta'$ at the base point $p$ of the tree satisfies $\EE_{\eta'} \deg(p) > 2$; see \cite[Theorems 6.1 and 6.2]{aldous2007processes}.  

The unimodular random tree $\eta'$ is edge-labeled; the labels are inherited from those of $\lambda$.
For each value of $0 < s < 1$  consider the unimodular random tree $\eta_s$ obtained by removing all edges $e$ with edge label $f(e) < s$ from a given $\eta'$-random tree, and then taking the connected component of the base point. We may fix $s$ to be sufficiently small so that $\EE_{\eta_s} \deg(p) >2$. The resulting random subtree is transient\footnote{In fact Proposition 4.9 of \cite{aldous2007processes} establishes the stronger statement that the speed of the simple random walk on a unimodular random tree $\eta$ with base point $p$ is given by $\mathbb{E}_{\eta} \frac{\mathrm{deg}(p)-2}{\deg (p)}$. This speed is positive if $\mathbb{E}_{\eta} \deg (p) > 2$.} with a positive $\eta_s$-probability, as follows from \cite[Proposition 4.9]{aldous2007processes}. Recall that $\eta_s$ is a random subgroup of $\lambda$. Hence a $\lambda$-random graph is transient with positive $\lambda$-probability by Remark \ref{remark:transient subgraph}.
\end{proof}


We now proceed with the following proof.

\begin{proof}[Proof of Theorem \ref{thm:infinitely ended unimodular random surfaces are transient}]
Let $\mu$ be a unimodular random hyperbolic manifold with infinitely many infinite-volume ends. We may assume without loss of generality that $\mu$ is extremal.

Let $\nu$ be a good and controlled point process over the unimodular random manifold $\mu$ (i.e. this is a random point configuration on a random manifold). Such a process exists by Corollary \ref{cor:exists good point process of unimodular}.  We know that a $\nu$-random manifold $(M,p;m)$ is transient with respect to Brownian motion if and only if the Delaunay graph $\DelanGraph (M,p;m)$ is $w$-transient with respect to a suitable edge weight function $w$ (Theorem \ref{thm:transience passes to good graphs}). However, the resulting random graph $\DelanGraph_* \nu$ is not in general unimodular. To work our way around this problem, we will consider thick graphs as an intermediate step. 

Let $\widehat{\nu}$ be the unimodular random metric measure space constructed from $\nu$ by conditioning it on the subset $\ThickHypersP$, and  $\lambda = \psi \cdot \ThickGraph_* (\widehat{\nu})$ be the random thick graph obtained from it. The random pointed graph $\lambda$ is unimodular; see Proposition \ref{prop:modified point process} and its proof for details. 

Note that  $\lambda$-almost every graph is infinitely ended by Lemma \ref{lemma:properties of thick graph}. Further, recall that $\lambda$ is regarded as a random $\left[0,1\right]$-edge-labeled graph, with positive labeling
representing the function $\tau$.  According to Lemma \ref{lemma:finding trees} there is a sufficiently small threshold $0 < s < 1$ such that the connected component of the base point $p$ in the subgraph of $G$ obtained by removing all edges $e \in E(G)$ with $f(e) < s$ is transient\footnote{By transient we mean with respect to the simple nearest neighbor random walk on the graph. By $w$-transient we mean transient with respect to the random walk corresponding to the edge labels defined by $w$.} for a positive $\lambda$-measure set of pointed $\left[0,1\right]$-edge labeled graphs $(G,p;f) \in \GraphsU$. Let $\eta_s$ denote the distribution of this transient pointed subgraph (so that $\eta_s$ is a positive measure, typically not a probability measure).

We turn our attention back to Delaunay graphs. For $\eta_s$-almost every transient subgraph $T$ of a generic thick graph $\ThickGraph(M,m)$, consider its preimage $\widetilde T$ in the corresponding Delaunay graph $\DelanGraph(M,m)$ with respect to the quotient map that was involved in Definition \ref{def:thick graph}. Namely $\widetilde{T}$ is a connected subgraph of $\DelanGraph(M,m)$ containing the base point. Observe that there is a uniform bound on the diameters of the fibers in the quotient map $\widetilde{T} \to T$ (since in the construction of $\eta_s$ we  maintain only those neighborhoods of short geodesics where $\tau > s$, and since the point process in question is controlled). In particular, this quotient map is a quasi-isometry. Hence the transience of the graph $T$ implies the transience of the graph $\widetilde T$ \cite[Theorem 3.10]{woess2000random}. As the edge weight function $w$ on the graph $\widetilde{T}$ is lower-bounded away from $0$ as well as upper bounded by $1$, we deduce from Lemma \ref{lemma:bounded w} that $\widetilde{T}$ is $w$-transient. It follows with a positive $\nu$-probability that the Delaunay graph $\DelanGraph(M,m)$ contains a $w$-transient connected subgraph $\widetilde T$ and therefore is itself $w$-transient by Remark \ref{remark:transient subgraph}. 

We conclude that $\mu$-almost every hyperbolic manifold is transient, by the above-mentioned Theorem \ref{thm:transience passes to good graphs} and by extremality.
\end{proof}

\begin{remark}
In case the unimodular random hyperbolic manifold $\mu$ admits a uniform lower bound on injectivity radius,  its Delaunay graphs are unimodular by Corollary \ref{cor:unimodular graph from unimodular space - delaunay}, and the above proof becomes much simpler. In the general case however, it may not be possible to define a \emph{finite} unimodular measure on the associated Delaunay graphs. This is where thick graphs \enquote{truncated} to a certain threshold become  useful.
\end{remark}

As a consequence of the above statement and its proof, we  obtain the following.

\begin{cor}
\label{cor:thick part is infinitely ended}
Let $\mu$ be a unimodular random hyperbolic manifold with infinitely many infinite-volume ends. There is a sufficiently small  $\varepsilon_\mu > 0$ such that with positive  $\mu$-probability the pointed manifold $(M,p) \in \Hypers$ has $\mathrm{InjRad}_M(p) \ge \varepsilon_\mu$ and  the connected component of the $\varepsilon_\mu$-thick part containing the base point $p$ is an infinitely-ended manifold (with boundary).
\end{cor}
\begin{proof}
This is a consequence of the previous Theorem    \ref{thm:infinitely ended unimodular random surfaces are transient} and its proof.
\end{proof}


\subsection{The one-ended case}
Our methods do not permit us to say much in the general and ubiquitous one-ended case. We do obtain the following criterion, which is essentially a reformulation of Theorem \ref{thm:main:general transient condition} of the introduction and of Theorem \ref{thm:transience passes to good graphs} above.

\begin{theorem}
\label{thm:criterion in one-ended case}
Let $M$ be an $n$-dimensional hyperbolic manifold and $\nu$ be a good point process on $M$. Then $M$ is transient with respect to Brownian motion if and only if the Delaunay graph $\DelanGraph(M,m)$ is $\nu$-almost surely $w$-transient with respect to the edge weight function $w = d_M^{n-2}$.
\end{theorem}

The same necessary and sufficient transience criterion applies in the probabilistic (rather than deterministic) setting. Namely, if $\mu$ is a unimodular random hyperbolic manifold and $\nu$ is a good point process over $\mu$ then $\mu$-almost every manifold $M$ is transient with respect to Brownian motion if and only if the Delaunay graphs $\DelanGraph(G,m)$ are almost surely $w$-transient where $w=d_M^{n-2}$.

We remark that a good point process  always process exists by Theorem \ref{thm:exists good point process} in the deterministic case and Corollary \ref{cor:exists good point process of unimodular} in the probabilistic case.

\section{The geodesic flow and the exit map}
\label{sec:ends and entropy}

For the purpose of this section, we find it more convenient to work with the dynamics of the geodesic flow rather than Brownian motion. These two notions are geometrically related. First, the Brownian motion is transient if and only if the geodesic flow is \cite{sullivan1981ergodic}. Second, a generic sample path of  Brownian motion sublinearly tracks a geodesic ray in $\HH^n$ \cite{blachere2011harmonic}. 
Lastly, the hitting measure induced by Brownian motion on the Gromov boundary $\partial \HH^n$ is equivalent to the Lebesgue measure \cite[Appendix B]{quint2006overview}. 
Hence, in the transient case, the Brownian motion and and the geodesic flow starting at a given base point induce equivalent hitting measures on the boundary at infinity.

We remark that by the classic Hopf--Tsuji--Sullivan dichotomy \cite{hopf1971ergodic,tsuji1959potential}, the geodesic flow on a given hyperbolic manifold is recurrent if and only if it is ergodic.

\subsection{The geodesic flow}
\label{sub:geodesic}

Consider the space $\HypersV$ of pointed hyperbolic manifolds with a unit tangent vector, namely
$$ \HypersV = \left\{(M,p,\vec v) \: : \: \begin{varwidth}{\linewidth}
\center
$(M,p) \in \Hypers$ and $\vec v$ is a unit tangent\\
vector to $M$ at the point $p$ 
\end{varwidth}\right\}.
$$
Given any  $(M,p,\vec v) \in \HypersV$ we let $ \gamma_t(M,p,\vec v) \in \HypersV$ denote the resulting geodesic flow, defined for all $t \in \RR$. 

A unimodular random hyperbolic manifold $\mu$ admits a canonical lift  $\vec \mu$ to a probability measure on $ \HypersV$. The measure $\vec \mu$ pushes forward to $\mu$ under the forgetful map $\HypersV \to \Hypers$ and disintegrates to the Lebesgue measure on the unit tangent sphere  over $\mu$-almost every point $(M,p) \in \Hypers$. 

\begin{lemma}
\label{lemma:vec mu is geodesic flow invariant}
Let $\mu$ be a unimodular random hyperbolic manifold. The probability measure $\vec \mu$ is invariant under the action induced by the geodesic flow.
\end{lemma}
\begin{proof}
The Chabauty space $\DSubtf{G}$ of discrete torsion-free subgroups of the simple Lie group $G = \mathrm{Isom}^+(\mathbb H^n) \cong \mathrm{PSO}(n,1)$ can be identified with the space of $n$-dimensional hyperbolic manifolds equipped with a base point and a tangent frame. 
Hence the space $\HypersV$ is naturally a quotient of the space $\DSubtf{G}$.
Let $\widetilde \mu$ be the canonical lift of $\mu$ to the space $\DSubtf{G}$, so that $\widetilde \mu$ pushes forward to $\mu$ under the forgetful map $\DSubtf{G} \to \Hypers$ and the  measure on $\mu$-almost fiber is invariant under the transitive action of the compact Lie group $\mathrm{SO}(n)$ on frames. 
The measure $\widetilde \mu$ is a unimodular random metric measure space, supported on homogeneous spaces of the form  $G/\Gamma$ equipped with the Haar measure. Hence the measure $\widetilde \mu$ is $G$-invariant under the natural conjugation action of $G$ on the space $\DSubtf{G}$, in other words, $\widetilde \mu$ is  an invariant random subgroup \cite[Theorem 1.4]{biringer2017unimodularity}. In particular $\widetilde \mu$ is invariant under the action of the Cartan subgroup $A$ corresponding to geodesic flow. Therefore the measure $\vec \mu$, which is obtained by pushing forward $\widetilde \mu$ via the $A$-equivariant map $\DSubtf{G} \to \HypersV$,  is  invariant under the geodesic flow. 
\end{proof}

Consider the following space 
$$\HypersV^{2} = \left\{(M,(p,\vec v),(q,\vec w)) \: :\: \text{\begin{varwidth}{\linewidth}
\center
    $(M,p,\vec v),(M,q,\vec w)\in \HypersV$ are tangent \\to the same directed geodesic on $M$
\end{varwidth}}\right\}.$$
The following lemma establishes a property which can be regarded as a \enquote{mass transport principle along the geodesic flow}.

\begin{lemma}
\label{lemma:MTP along geodesics}
Let $\mu$ be a unimodular random hyperbolic manifold. The two measures $L_{\vec \mu}$ and $R_{\vec \mu}$ on the space $\HypersV^2$ given by
$$ \mathrm{d}L_{\vec \mu}(M,(p,\vec v),\gamma_t(p,\vec v)) = \mathrm{d}\vec \mu(M,p,\vec v) \, \mathrm{d}\mathrm{Leb}_{\RR}(t) $$
and
$$ \mathrm{d}R_{\vec \mu}(M,\gamma_t(q,\vec w),(q,\vec w)) = \mathrm{d}\vec \mu(M,q,\vec w) \, \mathrm{d}\mathrm{Leb}_{\RR}(t) $$
satisfy $L_{\vec \mu} = R_{\vec \mu}$.
\end{lemma}
 \begin{proof}
 We know from the previous Lemma \ref{lemma:vec mu is geodesic flow invariant} that the measure $\vec \mu$ is invariant under the $\RR$-action corresponding to the geodesic flow. The space $\HypersV^2$ can be identified with the orbit equivalence relation for this action. It is a general statement in measure theory, that given a probability measure preserving action of a unimodular group, the \enquote{left} and \enquote{right} measures on the orbit equivalence relation coincide (as can be easily verified by means of Fubini's theorem). The desired result follows. 
\end{proof}

\subsection{The exit map}
\label{sub:exit map}

Consider a \emph{transient} hyperbolic manifold $M$. By the discussion in the beginning of \S\ref{sec:ends and entropy}, and by definition of the space of ends $\Ends{M}$, this means that for almost every base point and unit tangent vector $(p,v) \in T^1M$ the trajectory of the geodesic flow $\gamma_t(p,v)$  on $M$ converges to some end in $\Ends{M}$. This observation allows us to introduce the following.


\begin{definition}
\label{def:exit map}
Let $M$ be a transient hyperbolic manifold. The \emph{exit map} is a measurable map
$$\zeta_M : T^1M \to \Ends{M}$$ 
taking almost every base point and unit tangent vector $(p,\vec v) \in T^1 M$ to the end $\zeta_M(p,\vec v) \in \Ends{M}$ towards which the geodesic trajectory $\gamma_t(p,\vec v)$ accumulates as $t \to \infty$. Given \emph{any} base point $p \in M$ we may restrict the exit map $\zeta_M$ to the unit tangent sphere $T^1_pM$ and the obtain the measurable map $$\zeta_{M,p} : T^1_p M \to \Ends{M}.$$
\end{definition}

A finite-volume end of a hyperbolic manifold is a cusp. The subset of the boundary at infinity $\partial \HH^n$ corresponding to geodesics accumulating to such an end is countable. This shows that the image of the exit map $\zeta_M$ (or the map $\zeta_{M,p}$ for any point $p \in M$) is almost surely an infinite-volume end. Namely $\zeta_M : T^1 M \to \EndsInf{M}$ up to measure zero.

Let $p,q \in M$ be any two points on the same transient manifold $M$. The two probability measures on $\EndsInf{M}$ obtained by pushing forward the uniform probability measures on the respective unit tangent spheres $T^1_pM$ and $T^1_qM$ via the exit maps $\zeta_{M,p}$ and $\zeta_{M,q}$ are in the same measure class. In particular, the map $\zeta_{M,p}$ is essentially constant if and only if the map $\zeta_{M,q}$ is. 

\begin{remark}
\label{rmk:exit map from basepoint}
Let $\mu$ be a unimodular random hyperbolic manifold such that $\mu$-almost every manifold $M$ is transient.
By the \enquote{everything shows up at the root} principle (Lemma \ref{lemma: everything shows up at the root}) the following statements are equivalent:
\begin{itemize}
\item $\mu$-almost every pointed manifold $(M,p)$  admits a unique infinite-volume end $\zeta \in \EndsInf{M}$ such that $\zeta_M(M,q,\vec u) = \zeta$ for every base point and almost every unit tangent vector $(q,\vec u) \in M$, and
\item the exit map $\zeta_{M,p}$ is essentially constant.
\end{itemize}
\end{remark}

\begin{theorem}
\label{thm:random walk does not converge to a single end}
Let $\mu$ be a unimodular random hyperbolic manifold.  If $\mu$-almost every pointed manifold $(M,p)$ is transient and  satisfies $|\EndsInf{M}| > 1$ then for $\mu$-almost every pointed manifold $(M,p)$ the exit map $\zeta_{M,p} : T^1_p M \to \EndsInf{M}$  is not essentially constant.
\end{theorem}

Interestingly, we are obliged to provide  two different proofs for Theorem \ref{thm:random walk does not converge to a single end}, in the two-ended and the infinitely-ended cases (these two cases cover  all the possibilities by Theorem \ref{thm:Biringer--Raimbault}).

\begin{proof}[Proof of Theorem \ref{thm:random walk does not converge to a single end} in the infinitely-ended case]
Let $\mu$ be unimodular random hyperbolic manifold. Assume towards contradiction that $\mu$-almost every pointed hyperbolic manifold $M$ is transient, satisfies $|\EndsInf{M}| = \infty$,  and admits a unique  infinite-volume end $\zeta_M \in \EndsInf{M}$ such that the geodesic flow trajectory starting at every base point and almost every unit tangent vector $(q,\vec u) \in T^1M$ accumulates to $\zeta_M$.

For each $R > r > 0$ let $E_{r,R}$ be the following Borel set:
$$ E_{r,R} = \left\{ (M,x,y) \in \Hypers^{2} \: : \: 
\vcenter{
    \hbox{
      \begin{minipage}{7cm}
        \begin{enumerate}[label=\roman*)]
        \item $R \le d_M(x,y) \le R+1$,
          \item $M\setminus B(y,r)$ is not connected, and
          \item $x\in V$ for some unbounded connected component $V \in \mathcal{U}(B(y,r))$ which is \emph{not} a neighborhood of $\zeta_M$
        \end{enumerate}
      \end{minipage}
    }
  }\right\}.$$
The parameter $R$ will be chosen below (and will typically be much larger than $r$).

Assume that $(M,x_0,y_0) \in E_{r,R}$. We claim that any other point $y \in M$ with $(M,x_0,y) \in E_{r,R}$ must satisfy $d_M(y_0,y) < 2r+1$. Indeed, write $M \setminus B(y_0,r) = A \amalg B$ as a disjoint union of two collections of connected components, where $B$ is the connected end neighborhood of $\zeta_M$, and $A$ is the union of all other connected components. Namely  $A = M \setminus (B(y_0,r) \cup B)$ and $x_0 \in A$. There are several separate cases to consider:
\begin{itemize}
    \item Assume that $y \in B(y_0,r)$. Then the condition is clear, for certainly $d_M(y_0,y) \le r \le 2r+1$.
        \item Assume that $y \in A$. Let $V$ be the connected component of $M \setminus B_M(y,r)$ containing the point $x_0$. The condition $(M,x_0,y) \in E_{r,R}$ implies that $V$ is \emph{not} an end neighborhood of $\zeta_M$ and so the ball $B(y_0,r)$ cannot be entirely contained in $V$. Hence the ball $B(y,r)$ must intersect non-trivially either $B(y_0,r)$ or the geodesic arc $l$ from $x_0$ to $y_0$. In case $B(y,r) \cap B(y_0,r) \neq \emptyset$ the desired condition follows by the triangle inequality. In case $B(y,r) \cap l \neq \emptyset$ let $z$ be an arbitrary point of this intersection. Then $$d_M(x_0,y) \le d_M(x_0,z) + r.$$ 
    Since $d_M(x_0,y) \ge R$  we get $d_M(x_0,z) \ge R-r$. Since $d_M(x_0,y_0) \le R+1$ and as  $z$ lies on the geodesic arc from $x_0$ to $y_0$ we get $d_M(z,y_0) \le r+1$.  Therefore
    $$d_M(y_0,y) \le d_M(y_0,z) + d_M(z,y) \le (r+1)+r = 2r + 1$$
    as required.

    \item Assume that $y \in B$. The two points $x_0$ and $y$ lie in different connected components of $M \setminus B_M(y_0,r)$. Hence a geodesic path from $x_0$ to $y$ must intersect $B(y_0,r)$ at some point $y_1$.  We get
    $$ d_M(x_0,y_1) + d_M(y_1,y) = d_M(x_0,y) \le   R+ 1.$$ 
    The triangle inequality gives $d_M(x_0,y_1) \ge R-r$. Putting these two inequalities together gives $d_M(y,y_1) \le r+1$. Hence $d_M(y,y_0) \le 2r+1$ as required.

\end{itemize}

We now invoke the mass transport principle (Lemma \ref{lemma:MTP}). 
Let $F_{r,R} : \Hypers^2 \to \mathbb{R}$ be the characteristic function of the set $E_{r,R}$. 
The above claim gives the upper bound
$$ L_\mu (F_{r,R}) = \int_{\Hypers^2} F_{r,R}(M,x,y) \; \mathrm{dvol}_M(y) \; \mathrm{d}\mu(M,x) \le \mathrm{vol}_{\HH^n} B(2r+1).$$
We have used the fact that the hyperbolic volume of a ball in any hyperbolic manifold is upper bounded by the hyperbolic volume of the ball of the same radius in the hyperbolic space. The right-hand side is a constant depending only on $r$. 
We will arrive at a contradiction  by choosing $R$  to be sufficiently large, as follows. 

By Corollary \ref{cor:thick part is infinitely ended} there is an  $\varepsilon  >0$ such that the event $E$ defined by requiring that the pointed manifold $(M,p)$ has  $\mathrm{InjRad}_M(p) \ge \varepsilon$ and the connected component of the $\varepsilon$-thick part at the base point $p$ is infinitely-ended satisfies $\mu(E)>0$. 
As $r\to \infty$, the $\mu$-probability of the event $D$ that the ball $B_M(p,r)$ disconnects the random pointed manifold $(M,p)$ into at least two connected components tends to $1$.
Choose $r$ sufficiently large so that  $\mu(E\cap D) > 0$.

For each $N \in \NN$ we may find a radius $R_1(N) > 0$ such that the event $E_N$ given by requiring that $(M,p) \in E\cap D$ and that the space $M_{\ge \varepsilon} \setminus B(p,R_1(N))$ has at least $N$  connected components not accumulating to the end $\zeta_M$ satisfies $\mu(E_N) > \frac{1}{2}\mu(E\cap D)>0$.
Set $R = 
R(N) = R_1(N) + 2\varepsilon$. For each pointed manifold $(M,p) \in E_N$, it is possible to choose points $x_1,\ldots,x_N$ lying in each of those $N$ connected components. Note that $d_M(x_i,x_j) \ge 2\varepsilon$ for every  pair of distinct indices $i$ and $j$. Further $\mathrm{InjRad}_M(x_i) \ge \varepsilon$ for each $i$, since those points all lie in $M_{\ge \varepsilon}$. We obtain that $(M,z,p) \in E_{r,R}$ for all points $z \in \bigcup_{i=1}^N B(x_i,\varepsilon)$. This gives  the lower bound
$$ R_\mu(F_{r,R}) = \int_{\Hypers^2} F_{r,R}(M,x,y) \; \mathrm{dvol}_M(x) \; \mathrm{d}\mu(M,y)  \ge \mu(E_N)\cdot N \cdot \mathrm{vol}_{\HH^n} B(\varepsilon).$$
Note that $\mu(E_N) \cdot \mathrm{vol}_{\HH^n} B(\varepsilon)$ is bounded away from zero, while the parameter $N$ can be chosen arbitrarily large. We arrive at a contradiction to the mass transport principle by taking
 $$N \ge \frac{\mathrm{vol}_{\HH^n} B(2r+1)}{
 \tfrac12\mu(E\cap D) \cdot \mathrm{vol}_{\HH^n} B(\varepsilon)}$$
and choosing a sufficiently large radius $R = R(N)$ accordingly.
\end{proof}

The proof of Theorem \ref{thm:random walk does not converge to a single end} in the two-ended case will be completed in \S\ref{subsec:exit map for 2} below.

\subsection{The Liouville property}
\label{sub:Liouville}

Let $M$ be a hyperbolic manifold. A smooth function $f : M \to \mathbb{R}$ is called \emph{harmonic} if $\Delta h=0$. Equivalently, a function $f : M \to \mathbb{R}$ is harmonic if $$f =P_tf  \quad \text{where} \quad P_tf(x) = \int_M p(t,x,y)f(y)  \; \mathrm{vol}_M(y) \quad \forall x \in M $$
for all times $t \ge 0$. Here $p(t,x,y)$ is the transition probability of the Brownian motion starting at the point $x$ and at time $t$. The manifold $M$ is called \emph{Liouville} if all bounded harmonic functions on $M$ are constant. Our strategy will be to use the following.

\begin{theorem}[{Lessa \cite{lessa2016brownian}}]
\label{theorem:Lessa}
Let $\mu$ be a unimodular random hyperbolic manifold. If $\mu$-almost every manifold is non-Liouville then $\mu$-almost every manifold has positive volume growth and drift, and is in particular transient.
\end{theorem}
\begin{proof}[Proof outline from \cite{lessa2016brownian}]
Assume without loss of generality that $\mu$ is extremal. Lessa shows in \cite[Theorem 2.15]{lessa2016brownian} that the entropy $h$, speed $l$ and volume growth $v$ are well-defined  for $\mu$-almost every manifold, and satisfy the inequality $h \le lv$. The entropy $h$ is equal to zero if and only if $\mu$-almost every manifold is Liouville \cite[Theorem 2.11]{lessa2016brownian}. The desired conclusion follows\footnote{Lessa's work \cite{lessa2016brownian} is even more general, in that it applies to stationary (rather than just unimodular) random Riemannian manifolds, not necessarily hyperbolic.}.
\end{proof}

The following observation relates the notion of ends and harmonic functions.

\begin{lemma}
\label{lem:ends and Liouville}
Let $M$ be a transient hyperbolic manifold.
Assume that there is a Borel subset $\mathcal{E} \subset \EndsInf{M} $ such that for some (equivalently, any) point $p\in M$ the exit map $\zeta_{M,p}$ takes values in both sets $\mathcal{E}$ and $\EndsInf{M}\setminus \mathcal{E}$ with positive probability on $T^1_p M$. Then the function $f : M \to \left[0,1\right]$ given by
$$h(p) = \textrm{probability that Brownian motion starting at $p$ tends to $\mathcal{E}$}$$
for all $p \in M$
is a non-constant bounded harmonic function on $M$. In particular, the manifold $M$ is non-Liouville.
\end{lemma}
\begin{proof}
The fact that the function $h$ is harmonic follows by observing that $h = P_th$ for all times $t \ge 0 $. The function $h$ is certainly bounded in the range $\left[0,1\right]$. Let $p \in M$ be any base point. L\'evy's zero-one law states that the trajectory of the Brownian motion $X_t$ starting at the point $p$ almost surely satisfies $\lim_{t\to\infty}h(X_t) \in \{0,1\}$  \cite[Theorems 3.19 and 3.21]{LeGall2016}. It follows that the harmonic function $h$ is not constant.
\end{proof}

We deduce the following result, which was stated in the introduction and says that an infinitely-ended unimodular random hyperbolic manifold has positive drift.

\begin{proof}[Proof of Theorem \ref{thm:main infinitely ended}]
Let $\mu$ be a unimodular random hyperbolic manifold such that  $\mu$-almost almost every manifold has infinitely many infinite-volume ends. We know that $\mu$-almost every manifold is transient by Theorem \ref{thm:infinitely ended unimodular random surfaces are transient}. Further, we have established in Theorem \ref{thm:random walk does not converge to a single end} for $\mu$-almost every pointed manifold $(M,p)$ that the exit map $\zeta_{M,p}$ is not essentially constant. 
Namely $M$ admits a decomposition of its space of infinite-volume ends $\EndsInf{M} = \mathcal{E}_1(M) \amalg \mathcal{E}_2(M)$ as a disjoint union of two Borel subsets, such that Brownian motion  on the manifold $M$ starting at $p$ has a positive probability of accumulating to either subset $\mathcal{E}_1(M)$ or $\mathcal{E}_2(M)$. 
Hence $\mu$-almost every manifold $M$ admits a non-constant bounded harmonic function and is non-Liouville by Lemma \ref{lem:ends and Liouville}. We conclude that $\mu$-almost every manifold $M$ has positive drift by Theorem \ref{theorem:Lessa}.
\end{proof}

\begin{remark}
The reader is referred to e.g. \cite{benjamini2012ergodic,curien2017random} for the analogous theory of harmonic functions and the Liouville property for unimodular random graphs (which is not being used here directly).
\end{remark}

\subsection{The exit map for two ends}
\label{subsec:exit map for 2}

We pick up the remaining case of unimodular random hyperbolic manifolds with two infinite-volume ends using a completely different idea. It relies on the geometric behavior of the geodesic flow (see \S\ref{sub:geodesic} for details). We distinguish the following two types of pointed transient manifolds with a choice of a unit tangent vector $(M,p,\vec v) \in \HypersV$:
\begin{itemize}
    \item $(M,p,\vec v)$ is called \emph{unipolar} if the bi-infinite geodesic flow trajectory $\gamma_t(S,p,\vec v)$  tends to the same end as $t \to \pm \infty$, and
    \item $(M,p,\vec v)$ is called \emph{bipolar} if the bi-infinite geodesic flow trajectory $\gamma_t(S,p,\vec v)$ tends to different ends as $t \to \pm \infty$.
\end{itemize}

\begin{theorem}
\label{thm:two ended is bipolar}
Let $\mu$ be a unimodular random hyperbolic manifold. 
If $\mu$-almost every pointed hyperbolic manifold $(M,p) \in \Hypers$ is transient and satisfies $|\EndsInf{M}| = 2$ then $\vec \mu$-almost every $(M,p,\vec v) \in \HypersV$ is bipolar.
\end{theorem}
\begin{proof}
Consider the geodesic flow invariant probability 
$\vec \mu$ measure on the space $\HypersV$; see Lemma \ref{lemma:vec mu is geodesic flow invariant}.
The property of being unipolar or bipolar is invariant under the action of geodesic flow. Hence we may assume towards contradiction that the geodesic flow action is ergodic on $\vec \mu$ and that $\vec \mu$-almost every manifold with a base point and a unit tangent vector $(M,p,\vec v) \in \HypersV$ is transient, satisfies $|\EndsInf{M}|=2$ and is unipolar. 
This assumption implies by \Cref{rmk:exit map from basepoint} that $\vec \mu$-almost surely there exists an end $\zeta_M \in \EndsInf{M}$ such that $\lim_{t\to\pm\infty}\gamma_t(M,p,\vec v) = \zeta_M$.

Given a hyperbolic manifold $M$ with $|\EndsInf{M}|=2$ and a point $p\in M$ denote 
$$ \rho_M(p) = \inf \{ r > 0 \: : \:  \text{ $M \setminus B(p,r)$ has  two infinite-volume components} \}.$$ 
The quantity $\rho$ is measurable and $\vec \mu$-almost surely finite, since $\vec \mu$-almost every manifold has two infinite-volume ends. Now, consider a tangent unit vector $\vec v$ at the point $p$, and assume that $\gamma_t(M,p,\vec v) \to \zeta_M$ as $t \to \pm \infty$.
Consider the bi-infinite locally geodesic path  $\lambda_M(p,\vec v)$ on the manifold $M$ passing through $p$ in the direction $\vec v$. Define 
$$\tau_M(p,\vec v) = \inf_{q \in \lambda_M(p,\vec v)}  \rho_M(q) + 1$$
(where the $+1$ is taken as the infimum need not be achieved).
Define
$$X_M(p,\vec v) = \{ q \in \lambda_M(p,\vec v) \: : \: \rho_M(q) \le \tau_M(p,\vec v)\}.$$
By definition $X_M(p,\vec v) \neq \emptyset$ holds true $\vec \mu$-almost surely. 

Given each point  $x\in X_M(p,\vec v)$  consider the decomposition
$$X - B_M(x,\tau_M(p,\vec v)) = M_x^+ \amalg M_x^-$$ where $M_x^+$ and $M_x^-$ are two infinite-volume  connected components, such that $M_x^+$ is an end neighborhood of the end $\zeta_M$.
For a pair of points $x,y \in X_M(p,\vec v)$ we say that \emph{$x$ dominates $y$} and write $x \succ y$ if the two balls $B(x,\tau_M(p,\vec v))$ and $B(y,\tau_M(p,\vec v))$  are disjoint and $M_x^+ \supset M_y^+  $. The relation $\succ$ on the set $X_M(p,\vec v)$ has the following properties:
\begin{enumerate}[label=(\alph*)]
    \item $\succ$ is anti-reflexive, asymmetric and transitive. The asymmetry relies on the fact that $M$ has only two infinite-volume ends.
    \item For every $x \in X_M(p,\vec v)$ we have that $x \succ y$ for all $y=\gamma_t(M,p,\vec v) \in X_M(p,\vec v)$ provided $|t|$ is sufficiently large. 
    \item $x \succ y $ implies $d_M(x,y) \ge \tau_M(p,\vec v)$.
\end{enumerate}

We say that a point $y \in X_M(p,\vec v)$ is \emph{dominated} if $x \succ y$ for some $x \in X_M(p,\vec v)$. Denote $$Y_M(p,\vec v) = \{x \in X_M(p,\vec v) \: : \: \text{$x$ is \emph{not} dominated} \}.$$ 
Property (b) shows that the subset $Y_M(p,\vec v)$ is bounded (or what is equivalent, the set of  times that the geodesic $\lambda_M(p,\vec v)$ spends in $Y_M(p,\vec v)$ is bounded). We claim that $Y_M(p,\vec v) \neq \emptyset$.  Otherwise, every point $x \in X_M(p,\vec v)$ would be dominated, and we would be able to extract a sequence $x_n \in X_M(p,\vec v)$ with $x_{n+1} \succ x_n$. In particular $x_n \succ x_1$ for all $n$ by transitivity. On the other hand, properties (b) and (c) imply that $x_1 \succ x_n$ for all $n$ sufficiently large. This is a contradiction to property (a), namely the asymmetry of the relation $\succ$.

Finally, since $Y_M(p,\vec v)$ was defined in a geodesic flow equivariant way, it  depends only on the oriented locally geodesic path $\lambda_M(p,\vec v)$ and not on the base point $p$. In other words, the object $Y_M(p,\vec v)$ is invariant under the geodesic flow (it may depend on the choice of the future and past directions).

Define $\omega_M(p,\vec v) = \gamma_{T_M(p,\vec v)}(M,p,\vec v)$ where $T_M(p,\vec v)=\sup \{  t\; : \gamma_t(M,p,\vec v)\in   Y_M(p,\vec v)\}$, i.e. $\omega_M(p,\vec v)$ is the point in $Y_M(p,\vec v)$ with the supremum visitation time $t$ along the geodesic line $\lambda_M(p,\vec v)$. By the above reasoning, the assignment $(M,p,\vec v) \to \omega_M(p,\vec v)$ is invariant under geodesic flow. This sets up the objects required to arrive at a contradiction to the invariance of the measure $\vec \mu$ under the geodesic flow. Let us provide the details.

Consider the space $\HypersV^2$ introduced in \S\ref{sub:geodesic} above, and the  Borel function $F : \HypersV^{(2)} \to \left[0,1\right]$  given by
\begin{equation}
F(M,(p,\vec{v}),(q,\vec{u})) = \begin{cases}
1 & \text{$\gamma_t(M,q,\vec u) = \omega_M(p,\vec{v})$ for some $|t|\le 1$} \\
0 & \text{otherwise}.
\end{cases}
\end{equation}
Recall the left and right measures $L_{\vec \mu}$ and $R_{\vec \mu}$ introduced in Lemma \ref{lemma:MTP along geodesics}. Certainly $L_{\vec \mu}(F) = 2$. Conversely, by a suitable adaptation of the \enquote{everything shows up at the root} principle (i.e. Lemma \ref{lemma: everything shows up at the root}), we know that with positive $\vec \mu$-measure $(M,p,\vec v)$ satisfies  $\gamma_t (M,p,\vec v) = \omega_M(p,\vec v)$ for some $|t| \le 1$. 
Hence   $R_{\vec \mu}(F) = \infty$. This is a contradiction to the invariance of $\vec \mu$ under the geodesic flow, and more specifically, to Lemma \ref{lemma:MTP along geodesics}.
\end{proof}

We may now complete the proof Theorem \ref{thm:random walk does not converge to a single end} in the remaining two-ended case.  

\begin{proof}[Proof of Theorem \ref{thm:random walk does not converge to a single end} in the two-ended case]
Let $\mu$ be a unimodular hyperbolic manifold. Assume towards contradiction that $\mu$-almost every pointed manifold $(M,p)$ is transient, satisfies $|\EndsInf{M}|=2$ and the exit map $\zeta_{M,p}$ is essentially constant on $T^1_p M$. Let $\vec \mu$ be the lift of $\mu$ to a probability measure on $\HypersV$ invariant under the action of the geodesic flow (see the discussion in \S\ref{sub:geodesic}). Note that $\vec \mu$-almost every $(M,p,\vec v)$ is unipolar. We arrive at a contradiction to Theorem \ref{thm:two ended is bipolar}.
\end{proof}

We are ready to present the following.

\begin{proof}[Proof of Theorem \ref{thm:main two ended}]
Let $\mu$ be a unimodular random hyperbolic manifold, such that $\mu$-almost every manifold $M$ has exactly two infinite-volume ends $\zeta_M$ and $\zeta'_M$. 
Our goal is to show that $\mu$-almost every manifold is recurrent. Assume towards contradiction that $\mu$-almost every manifold is transient. 
We will be relying on the geodesic flow and on the exit map; see the discussions in \S\ref{sub:geodesic} and \S\ref{sub:exit map}.

We know by Theorem \ref{thm:random walk does not converge to a single end} that for $\mu$-almost every pointed hyperbolic manifold $(M,p)$ the exit map $\zeta_{M,p}$ is not essentially constant on $T^1_pM$.  This allows us to apply Lemma \ref{lem:ends and Liouville} and define a non-constant bounded harmonic function $h_M$ on the manifold $M$ by letting $h_M(q)$  be the probability that the Brownian motion starting at the point $q \in M$ tends to a specific infinite-volume end $\zeta_M \in \EndsInf{M}$. 

Given a pointed manifold $(M,p)\in \Hypers$ denote by $\theta_p$ the uniform probability measure on the unit tangent sphere $T^1_pM$. 
By the construction of the non-constant harmonic function $h_M$ and by the \enquote{everything shows up at the root} principle (Lemma \ref{lemma: everything shows up at the root}), we see that for a positive $\mu$-measure of pointed manifolds $(M,p)$ the push-forward of the measure $\theta_p$ via the exit map $\zeta_{M,p}$ is different than $\frac{1}{2}\delta_{\zeta_M}+\frac{1}{2} \delta_{\zeta'_M}$. In other words, the probability that the Brownian motion starting at $p$ accumulates to either end $\zeta_M$ or $\zeta'_M$ is not the same.

Lift $\mu$ to a geodesic flow invariant probability measure $\vec \mu$ on the space $\HypersV$. 
There is a symmetric involution $I_p$ defined on each unit tangent sphere $T^1_pM$ and given by $I_p(v)=-v$ for all $v \in T^1_pM$. The involution $I_p$ preserves the measure $\theta_p$. 
From the discussion in the previous paragraph regarding the push-forward of $\theta_p$ via the exit map, and by considering the action of the involution $I_p$ on the unit tangent sphere $T^1_pM$, we deduce that $(M,p,\vec v)$ is unipolar with a positive $\vec \mu$-probability. This is a contradiction to Theorem \ref{thm:two ended is bipolar}.
\end{proof}

\bibliographystyle{alpha}
\bibliography{unimodular}

\end{document}